\documentclass[11pt, english, draft]{article}
\usepackage{amssymb,amsmath}
\title{\bf OD-Characterization of Certain Four Dimensional Linear Groups with  Related Results
Concerning Degree Patterns\thanks{This work has been supported by
{\bf RIFS}.}}
\author{ {\bf B. Akbari} and {\bf A. R. Moghaddamfar}\\[0.1cm]
{\em Department of Mathematics, K. N. Toosi
University of Technology,}\\
 {\em P. O. Box $16315$-$1618$, Tehran, Iran}\\[0.1cm]
 {\em and} \\[0.1cm]
{\em Research Institute for Fundamental Sciences (RIFS), Tabriz, Iran}\\[0.1cm]
{\em E-mails}: {\tt moghadam@kntu.ac.ir} and {\tt
moghadam@ipm.ir}}

\newenvironment{proof}{\noindent {\em {Proof}}.}{$\square$
\medskip}
\newtheorem{definition}{Definition}[section]

\newtheorem{theorem}{Theorem}[section]
\newtheorem{proposition}{Proposition}[section]

\newtheorem{lm}{Lemma}[section]

\newtheorem{problem}{Open Problem}[section]
\begin{document}
\maketitle
\begin{abstract}
\noindent The prime graph of a finite group $G$, which is denoted
by ${\rm GK}(G)$, is a simple graph whose vertex set is comprised
of the prime divisors of $|G|$ and two distinct prime divisors $p$
and $q$ are joined by an edge if and only if there exists an
element of order $pq$ in $G$. Let $p_1<p_2<\cdots<p_k$ be all
prime divisors of $|G|$. Then the degree pattern of $G$ is
defined as ${\rm D}(G)=(\deg_G(p_1), \deg_G(p_2), \ldots,
\deg_G(p_k))$, where $\deg_G(p)$ signifies the degree of the
vertex $p$ in ${\rm GK}(G)$. A finite group $H$ is said to be
OD-characterizable if $G\cong H$ for every finite group $G$ such
that $|G|=|H|$ and ${\rm D}(G)={\rm D}(H)$. The purpose of this
article is threefold. First, it finds sharp upper and lower
bounds on $\vartheta(G)$, the sum of degrees of all vertices in
${\rm GK}(G)$, for any finite group $G$ (Theorem \ref{p2}).
Second, it provides the degree of vertices $2$ and the
characteristic $p$ of the base field of any finite simple group
of Lie type in their prime graphs (Propositions
\ref{prop-Clasic-1}-\ref{suzuki}). Third, it proves the linear
groups $L_4(19)$, $L_4(23)$, $L_4(27)$, $L_4(29)$, $L_4(31)$,
$L_4(32)$ and $L_4(37)$
are OD-characterizable (Theorem \ref{main-2}).\\[0.3cm]
\end{abstract}
\renewcommand{\baselinestretch}{1.1}
\def\thefootnote{ \ }
\footnotetext{{\em $2010$ Mathematics Subject Classification}:
20D05, 20D06, 20D08.}
\section{Introduction}
All the groups under consideration are finite and simple groups
are non-abelian. For a natural number $n$, we denote by $\pi(n)$
the set of prime divisors of $n$ and put $\pi(G)=\pi(|G|)$. The
{\em spectrum} ${\omega}(G)$ of a group $G$ is the set of orders
of all elements in $G$. The set ${\omega}(G)$ determines the
prime graph (or Gruenberg-Kegel graph) ${\rm GK}(G)$ whose vertex
set is $\pi(G)$ and two vertices $p$ and $q$ are {\em adjacent}
if and only if $pq\in {\omega}(G)$. Denote by $s=s(G)$ the number
of connected components of ${\rm GK}(G)$ and by $\pi_i(G)=\pi_i$
$(i= 1, 2, \ldots, s)$, its $i$th connected component. When $G$
has even order we assume that $\pi_1$ is the connected component
containing the prime 2. Denote by $\mu(G)$ the set of numbers in
${\omega}(G)$ that are maximal with respect to divisibility
relation. Observe that $\omega(G)$ is uniquely restored from
$\mu(G)$. Denote by $\omega_i(G)$, the set consisting of $n\in
\omega(G)$ such that every prime divisor of $n$ lies in $\pi_i$.
The connected components of all finite simple groups are obtained
in \cite{w} and \cite{k}.

The {\em degree ${\rm deg}_G(p)$ of a vertex} $p\in \pi(G)$ is
the number of edges incident on $p$. When there is no ambiguity,
we will use ${\rm deg}(p)$ instead of ${\rm deg}_G(p)$. Assume
that $\pi(G)=\{p_1, p_2, \ldots, p_k\}$ with
$p_1<p_2<\cdots<p_k$. We put
\[ {\rm D}(G):=\big({\rm deg}_G(p_1), {\rm deg}_G(p_2),\ldots, {\rm deg}_G(p_k)\big), \]
which is called the {\em degree pattern of $G$}, and
$$\vartheta(G)=\sum_{i=1}^k\deg_G(p_i).$$
Given a finite group $M$, denote by $h_{\rm OD}(M)$ the number of
isomorphism classes of finite groups $G$ such that $|G|=|M|$ and
${\rm D}(G)={\rm D}(M)$. It is clear that $1\leqslant h_{\rm
OD}(M)<\infty$ for any finite group $M$. In terms of function
$h_{\rm OD}$, the groups $M$ are classified as follows:

\begin{definition} A finite group $M$ is called a {\em $k$-fold {\rm
OD}-characterizable} group if $h_{\rm OD}(M)=k$. Usually, a
$1$-fold OD-characterizable group is simply called a {\em {\rm
OD}-characterizable} group.
\end{definition}

In this article, we first derive sharp upper and lower bounds on
$\vartheta(G)$ for any finite group $G$. We also present some
examples which illustrate our results (Section 2). We also
determine $\deg(2)$ and $\deg(p)$ in the prime graph associated
with any finite simple group of Lie type over a field of
characteristic $p$ (Section 3). In Section 4, we will focus our
attention on the OD-characterizability of the projective special
linear groups $L_4(q)$ for certain $q$. Previously, it was proved
in \cite{bakbari}, \cite{amr}, \cite{mz1} and \cite{ls} that the
projective special linear groups $L_4(2)$, $L_4(3)$, $L_4(4)$,
$L_4(5)$, $L_4(7)$, $L_4(8)$, $L_4(9)$, $L_4(11)$, $L_4(13)$,
$L_4(16)$ and $L_4(17)$ are OD-characterizable. In this article,
we will prove that the simple groups $L_4(19)$, $L_4(23)$,
$L_4(27)$, $L_4(29)$, $L_4(31)$, $L_4(32)$ and $L_4(37)$ are
characterizable through their orders and degree patterns. Finally,
we will prove the following:\\[0.3cm]
{\bf Theorem A.} {\em   The projective special linear groups
$L_4(19)$, $L_4(23)$, $L_4(27)$, $L_4(29)$, $L_4(31)$, $L_4(32)$
and $L_4(37)$
are OD-characterizable.}\\[0.3cm]
\indent From this theorem the following corollary is derived.\\[0.1cm]
{\bf Corollary A.} {\em The projective special linear groups
$L_4(q)$, with $2\leqslant q<40$, are OD-characterizable.}\\[0.1cm]

Finally, a list of finite almost simple groups which are
presently known to be OD-characterizable or $k$-fold
OD-characterizable for $k\geq 2$ is given in Tables 3-4 under
Section 5.

{\em Notation and Terminology.} Given a connected graph
$\Gamma=(V, E)$, a {\em spanning tree} of $\Gamma$ is a connected
subgraph $T=(V, \tilde{E})$ of $\Gamma$ such that
$|\tilde{E}|=|V|-1$. A {\em clique} in a graph is a set of
pairwise adjacent vertices. A simple graph, in which every pair of
distinct vertices are adjacent is called a complete graph. We will
denote by $C_n$ the cycle on $n$ vertices and by $K_n$ the
complete graph on $n$ vertices.

Let $\Gamma=(V, E)$ be a simple graph. For any vertex $v\in V$,
the degree of $v$, $\deg_\Gamma(v)$, is the number of edges
containing $v$. The {\em degree sequence} of a graph is the
sequence of the degrees of its vertices, sometimes written in a
non-decreasing order, as $D(\Gamma): d_1\leqslant d_2\leqslant
\cdots \leqslant d_n$. A sequence of non-negative integers $(a_1,
a_2, \ldots, a_k)$ is said to be {\em majorized} by another such
sequence $(b_1, b_2, \ldots, b_k)$ if $a_i\leqslant b_i$ for
$1\leqslant i\leqslant k$. A graph $\Gamma_1$ is {\em
degree-majorized} by a graph $\Gamma_2$ if
$V(\Gamma_1)=V(\Gamma_2)$ and the non-decreasing degree sequence
of $\Gamma_1$ is majorized by that of $\Gamma_2$. For example,
$C_4$ is degree-majorized by $K_4$ because $(2, 2, 2, 2)$ is
majorized by $(3, 3, 3, 3)$. For a graph $\Gamma$ with vertex set
$V$, we set $\vartheta(\Gamma):=\sum_{v\in V}\deg_\Gamma(v)$. In
particular, for convenience, we will denote $\vartheta({\rm
GK}(G))$ as $\vartheta(G)$. The union of simple graphs
$\Gamma_1=(V_1, E_1)$ and $\Gamma_2=(V_2, E_2)$ is the graph
$\Gamma_1\cup\Gamma_2$ with vertex set $V_1\cup V_2$ and edge set
$E_1\cup E_2$. If $\Gamma_1$ and $\Gamma_2$ are disjoint (we
recall that two graphs are disjoint if they have no vertex in
common), we refer to their union as a disjoint union, and
generally denote it by $\Gamma_1\oplus\Gamma_2$. If $U\subseteq
V$ is any set of vertices, then $\Gamma[U]$ denotes the induced
subgraph of $\Gamma$ whose vertex set is $U$ and whose edges are
precisely the edges of $\Gamma$ which have both ends in $U$. A
set of vertices $I\subseteq V$ is said to be an independent set
of $\Gamma$ if no two vertices in $I$ are adjacent in $\Gamma$.
The independence number of $\Gamma$, denoted by $\alpha(\Gamma)$,
is the maximum cardinality of an independent set among all
independent sets of $\Gamma$. Given a group $G$, for convenience,
we will denote $\alpha({\rm GK}(G))$ as $t(G)$. Moreover, for a
vertex $r\in \pi(G)$, let $t(r,G)$ denote the maximal number of
vertices in independent sets of ${\rm GK}(G)$ containing $r$. For
a natural number $m$, the alternating and symmetric group of
degree $m$ denoted by $\mathbb{A}_m$ and $\mathbb{S}_m$,
respectively. Also, we will denote by $l_m$ the largest prime not
exceeding $m$. Given a prime $p$, denote by $m_p$ the $p$-part of
$m$, that is the largest power of $p$ diving $m$. For finite
simple groups, we borrow the notation from \cite{atlas}. We
denote by $H:K$ (resp. $H\cdot K$) a split extension (resp. a
non-split extension) of a normal subgroup $H$ by another subgroup
$K$. Note that, split extensions are the same as semi-direct
products.

If $a$ is a natural number, $r$ is an odd prime, and $(a, r)=1$,
then $e(r, a)$ denotes a multiplicative order of $a$ modulo $r$,
i.e., a minimal natural number $m$ with $a^m \equiv 1\pmod{r}$.
In the case when $a$ is an odd number, we set
$$ e(2, a)=\left\{\begin{array}{llll} 1 & & \mbox{if}& a \equiv 1\pmod{4},\\[0.1cm]
2 & & \mbox{if}&a \not\equiv 1\pmod{4}.
\end{array}\right.$$

We define two functions $\nu$ and $\eta$ on $\mathbb{N}$ as
follows:
$$\nu(m)=\left\{ \begin{array}{lll} m & \mbox{if} & m\equiv 0\!\!\!\pmod{4},
\\[0.1cm]
\frac{m}{2} & \mbox{if} & m\equiv 2\!\!\!\pmod{4},
\\[0.1cm]
2m& \mbox{if} & m \equiv 1\!\!\!\pmod{2},\\ \end{array} \right. \
\ \ \ \ \ {\rm and} \ \ \ \ \ \eta(m)=\left\{ \begin{array}{lll}
m & \mbox{if} & m\equiv 1\!\!\!\pmod{2},
\\[0.1cm]
\frac{m}{2}& \mbox{if} & m \equiv 0\!\!\!\pmod{2}.\\ \end{array}
\right.
$$ Given a prime $p$, we write $p^n\parallel m$ to indicate
that $p^n$ is the highest power of $p$ dividing $m$, in other
words, $p^n$ divides $n$ but $p^{n+1}$ does not. All further
unexplained notation is standard and can be found in \cite{atlas}
or \cite{rose}, for instance.
\section{Some Properties of Degree Pattern}
In \cite{suz2}, Suzuki studied the structure of a prime graph
associated with  a finite simple group, and proved the following
interesting result. It is worth stating that to prove this result
he does not use the classification of finite simple groups.
\begin{lm}[Suzuki] (see \cite[Theorem B]{suz2})\label{prop1}   Let $G$ be a
finite simple group whose prime graph ${\rm GK}(G)$ is
disconnected and let $\Delta$ be a connected component of ${\rm
GK}(G)$ whose vertex set does not contain $2$. Then $\Delta$ is a
clique.
\end{lm}

{\em Remark $1$.} A similar result as Lemma \ref{prop1} can be
found in \cite[Lemma 4]{km}.

{\em Remark $2$.} Note that Lemma $\ref{prop1}$ is true for all
{\em finite groups} not only finite simple groups. As a matter of
fact, if $G$ is a finite group with disconnected graph ${\rm
GK}(G)$ then one of the following holds (see \cite[Theorem A]{w}):
\begin{itemize}
\item[$(1)$] $s(G)=2$, $G$ is a Frobenius group.
\item[$(2)$] $s(G)=2$, $G$ is a 2-Frobenius group, i.e., $G=ABC$ where $A$,
$AB$ are normal subgroups in $G$, $B$ is a normal subgroup in
$BC$, and $AB$, $BC$ are Frobenius groups.
\item[$(3)$] There exists a non-abelian simple group $P$ such that $P\leqslant \bar{G}=G/N\leqslant {\rm Aut}(P)$
for some nilpotent normal $\pi_1(G)$-subgroup $N$ of $G$ and
${\bar G}/P$ is a $\pi_1(G)$-group. Moreover, ${\rm GK}(P)$ is
disconnected, $s(P)\geqslant s(G)$ and for every $i$,
$2\leqslant  i\leqslant s(G)$, there exist $j$, $2\leqslant
j\leqslant s(P)$, such that $\omega_i(G)=\omega_j(P)$.
\end{itemize}

\cite[Lemma 5]{mazurov-2002}: First assume that $G$ satisfies
$(1)$, that is, $G=KC$ is a Frobenius group with kernel $K$ and
complement $C$. Then ${\rm GK}(K)$ and ${\rm GK}(C)$ are
connected components of ${\rm GK}(G)$. Since $K$ is nilpotent,
${\rm GK}(K)$ is always complete. Moreover, either $C$ is
solvable and ${\rm GK}(C)$ is complete, or $C$ contains a normal
subgroup $L\cong {\rm SL}_2(5)$ such that $(|L|, |C:L|)\leqslant
2$ and ${\rm GK}(C)$ can be obtained from the complete graph on
$\pi(C)$ by deleting the edge $\{3, 5\}$. Thus, in all cases,
$\pi_2(G)$ is clique, as required.

\cite[Lemma 7(1)]{mazurov-2002}: \ Next suppose that $G$ satisfies
(2). In this case, ${\rm GK}(B)$ and ${\rm GK}(AC)$ are connected
components of ${\rm GK}(G)$ and are both complete graphs.
Therefore, the connected components $\pi_1(G)$ and $\pi_2(G)$ in
this case are cliques, as required.

Finally, if condition $(3)$ holds, then it follows by Lemma
\ref{prop1} that the connected components $\pi_j(P)$ for
$2\leqslant j\leqslant s(P)$ and so the connected components
$\pi_j(G)$ for $2\leqslant j\leqslant s(G)$ are cliques. This
completes our claim.

\begin{theorem}\label{p2} For any finite group $G$, we have
\begin{equation}\label{e1-2012}\sum_{i=1}^sn_i(n_i-1)-(n_1-1)(n_1-2)\leqslant\vartheta(G)\leqslant\sum_{i=1}^sn_i(n_i-1).\end{equation} where $s=s(G)$ and
$n_i=|\pi_i|$, $i=1, 2, \ldots, s$. In particular, if ${\rm
GK}[\pi_1]$ is complete, then we have
$$ \vartheta(G)=\sum_{i=1}^sn_i(n_i-1).$$
\end{theorem}
\begin{proof}
First of all, Lemma $\ref{prop1}$ and Remark 1 show that the
prime graph of an arbitrary finite group $G$ has the following
form:
$${\rm GK}(G)={\rm GK}[\pi_1]\oplus K_{n_2}\oplus \cdots\oplus K_{n_s},$$
where ${\rm GK}[\pi_1]$ denotes the induced subgraph ${\rm
GK}(G)[\pi_1]$, $n_i=|\pi_i|$ and $s=s(G)$. On the one hand, it
is clear that the prime graph ${\rm GK}(G)$ is degree-majorized
by the graph $K_{n_1}\oplus K_{n_2}\oplus \cdots\oplus K_{n_s}$.
Therefore, we obtain
\begin{equation}\label{e1}
\vartheta(G)\leqslant\vartheta(K_{n_1}\oplus K_{n_2}\oplus
\cdots\oplus K_{n_s})=\sum_{i=1}^sn_i(n_i-1).
\end{equation}
Furthermore, equality can only hold in Eq. (\ref{e1}) if  ${\rm
GK}[\pi_1]$ is complete. On the other hand, since ${\rm
GK}[\pi_1]$ is a connected subgraph of ${\rm GK}(G)$, it has a
spanning tree, say $T$. Obviously $\vartheta(T)=2(n_1-1)$. Now,
since $$T\oplus K_{n_2}\oplus \cdots\oplus K_{n_s}\subseteq {\rm
GK}[\pi_1]\oplus K_{n_2}\oplus \cdots\oplus K_{n_s}={\rm
GK}(G),$$ we conclude that
$$\begin{array}{lll}
\vartheta(G)&\geqslant& \vartheta(T\oplus K_{n_2}\oplus
\cdots\oplus
K_{n_s})\\[0.1cm]
& = & \vartheta(T)+\vartheta(K_{n_2}\oplus \cdots\oplus
K_{n_s})\\[0.1cm]
& = & 2(n_1-1)+\sum\limits_{i=2}^sn_i(n_i-1)\\[0.1cm]
& = & \sum\limits_{i=1}^sn_i(n_i-1)-(n_1-1)(n_1-2),
\end{array} $$
and the proof is complete. \end{proof}

Here, we present some examples to illustrate lower and upper
bounds in (\ref{e1-2012}) are sharp.

{\em Some Examples.} $(1)$ Let $G$ be a nilpotent group with
$|\pi(G)|=n_1$. Then the prime graph ${\rm GK}(G)$ is a complete
graph, $s(G)=1$ and so $\vartheta(G)=n_1(n_1-1)$, which shows
that the upper bound in (\ref{e1-2012}) is sharp. Of course, it is
easy to see that there are many non-nilpotent groups with complete
prime graph, for example, take $G=\mathbb{A}_5\times
\mathbb{A}_5$.

$(2)$ In \cite[Theorem 1]{lm}, a description is presented for all
finite non-abelian simple groups whose prime graph connected
components are cliques (for a revised list see \cite[Corollary
7.6]{vv}). The list of such groups is given in Table 1. If $G$
is  any of the groups listed in Table 1, then we have
$$\vartheta(G)=\sum_{i=1}^sn_i(n_i-1),$$ which is the upper bound in (\ref{e1-2012}).
\begin{center}
{\small {\bf Table 1.} \ {\it Finite non-abelian simple groups
whose prime
graph connected components are cliques }\\[0.3cm]
$
\begin{array}{llllll}
\hline
{\rm Group} & \mbox{restrictions} & \pi_1 & \pi_2 & \pi_3 & \pi_4\\
\hline \mathbb{A}_5, \mathbb{A}_6 & & \{2\}& \{3\}& \{5\}&\\[0.1cm]
\mathbb{A}_7 & & \{2, 3\}& \{5\}& \{7\} &
\\[0.1cm]
\mathbb{A}_9 & & \{2, 3, 5\}& \{7\}& & \\[0.1cm]
\mathbb{A}_{12} & & \{2, 3, 5, 7\}& \{11\}& & \\[0.1cm]
\mathbb{A}_{13} & & \{2, 3, 5, 7\}& \{11\}& \{13\} &
\\[0.1cm]
M_{11} & &\{2, 3\} & \{5\} & \{11\} &\\[0.1cm]
M_{22} & &\{2, 3\} & \{5\}& \{7\} & \{11\}\\[0.1cm]
J_1 & & \{2, 3, 5\}& \{7\}& \{11\} & \{19\}\\[0.1cm]
J_2 & & \{2, 3, 5\}& \{7\} & &\\[0.1cm]
J_3 & & \{2, 3, 5\} & \{17\} & \{19\}&\\[0.1cm]
HiS & & \{2, 3, 5\} & \{7\} & \{11\}&\\[0.1cm]
A_1(q) & q\equiv 1\pmod{4} & \pi(q-1) &  \{p\} & \pi(\frac{q+1}{2}) & \\[0.1cm]
A_1(q) & q\equiv -1\pmod{4}
& \pi(q+1) &  \{p\} & \pi(\frac{q-1}{2}) & \\[0.1cm]
A_1(q) & 2<q\equiv 0 \pmod{2} & \{2\} & \pi(q-1) & \pi(q+1) & \\[0.1cm]
A_2(4) & & \{2\}& \{3\} & \{5\}& \{7\}\\[0.1cm]
A_2(q) & (q-1)_3\neq 3,  \ q+1=2^k &  \pi(q(q^2-1))& \pi(\frac{q^2+q+1}{(3,q-1)}) &&\\[0.1cm]
{^2A}_3(2) & & \{2, 3\} & \{5\} & &\\[0.1cm]
{^2A}_5(2) & & \{2, 3, 5\} & \{7\} & \{11\} &\\[0.1cm]
{^2A}_2(q) & (q+1)_3\neq 3, \  q-1=2^k & \pi(q(q^2-1))& \pi(\frac{q^2-q+1}{(3,q+1)})& &\\[0.1cm]
C_3(2) & & \{2, 3, 5\} & \{7\} & & \\[0.1cm]
C_2(q) & q>2& \pi(q(q^2-1))& \pi(\frac{q^2+1}{(2, q-1)})&&\\[0.1cm]
D_4(2) & & \{2, 3, 5\} & \{7\} &&\\[0.1cm]
{^3D}_4(2) & & \{2, 3, 7\}& \{13\}& &\\[0.1cm]
G_2(q) & q=3^k& \pi(q(q^2-1)) & \pi(q^2-q+1) &
\pi(q^2+q+1) & \\[0.1cm]
{^2B}_2(q) & q=2^{2k+1}>1 & \{2\}& \pi(q-1) & \pi(q-\sqrt{2q}+1) & \pi(q+\sqrt{2q}+1)\\[0.3cm]
\hline
\end{array}
$}
\end{center}

$(3)$ Let $G$ be a Frobenius group with kernel $K$ and complement
$C$. Then ${\rm GK}(K)$ and ${\rm GK}(C)$ are connected components
of ${\rm GK}(G)$, and so ${\rm GK}(G)={\rm GK}(K)\oplus {\rm
GK}(C)$. Moreover, ${\rm GK}(K)$ is always complete because $K$ is
nilpotent, and ${\rm GK}(C)$ is complete if $C$ is solvable,
otherwise ${\rm GK}(C)$ can be obtained from the complete graph
on $\pi(C)$ by deleting the edge $\{3, 5\}$ (\cite[Lemma
5]{mazurov-2002}):
$${\rm GK}(C)=\left\{\begin{array}{ll} K_{|\pi(C)|} & \mbox{if} \  C \
\mbox{is
solvable,}\\[0.3cm] K_{|\pi(C)|}\setminus \{3, 5\} & \mbox{otherwise.} \end{array}
\right.$$ Let $\{n_1, n_2\}=\{|\pi(K)|, |\pi(C)|\}$. By what
observed above, we conclude that
$$\vartheta(G)=\left\{\begin{array}{ll} \sum_{i=1}^2n_i(n_i-1) & \mbox{if} \  C \
\mbox{is
solvable,}\\[0.3cm] \sum_{i=1}^2n_i(n_i-1)-2 & \mbox{otherwise.} \end{array}
\right.$$ Therefore, in the case $G$ is solvable, the value of
$\vartheta(G)$ is equal to the upper bound in (\ref{e1-2012}). On
the other hand, if $G$ is non-solvable, then $C$ contains a normal
subgroup $C_0$ of index $\leqslant 2$ such that $C_0\cong Z\times
{\rm SL}(2,5)$, where every Sylow subgroup of $Z$ is cyclic and
$\pi(Z)\cap \pi(30)=\emptyset$. Therefore, in the case when
$|\pi(C)|=n_1=3$, we obtain
$$\vartheta(G)= n_2(n_2-1)+4=\sum_{i=1}^2n_i(n_i-1)-(n_1-1)(n_1-2),$$
which shows that the lower bound in (\ref{e1-2012}) is sharp.
\section{The Degree of $2$ and the Characteristic in the Prime Graph of Simple Groups of Lie Type}
Let $G$ be a finite group with
$$|G|=p_1^{\alpha_1}p_2^{\alpha_2}\cdots p_k^{\alpha_k},$$
where $p_1<p_2<\cdots<p_k$ are primes and $\alpha_1, \alpha_2,
\ldots, \alpha_k, k$ are natural numbers. In many cases, the
degree pattern  $${\rm D}(G)=\big(\deg_G(p_1), \deg_G(p_2),
\ldots, \deg_G(p_k)\big),$$ of $G$, gives us more information
about the structure of $G$ or its prime graph ${\rm GK}(G)$. We
illustrate this with the following easy observations:
\begin{itemize}
\item[$(1)$] If $\deg_G(p_i)=0$ for some $i$, then $\{p_i\}$ is a
connected component of the prime graph ${\rm GK}(G)$ and $G$ is a
$C_{p_i, p_i}$-group which means that the centralizer of any
non-trivial $p_i$-element in $G$ is a $p_i$-group. Moreover, if
we put $\Omega_l(G):=\{p_i\in \pi(G) \ | \ \deg_G(p_i)=l\}$,
$0\leqslant l \leqslant k-1$, then $|\Omega_0(G)|\leqslant
s(G)\leqslant 6$ (see \cite{w}). Also, if $|\Omega_0(G)|\geqslant
3$, then $G$ is a non-solvable group (see \cite[Lemma 8]{lm}).
The same is true if $|\Omega_0(G)|=2$ and $|\pi(G)|\geqslant 3$.
\item[$(2)$]  If $\deg_G(p_i)=k-1$ for some $i$, then the prime graph ${\rm
GK}(G)$ is connected. Moreover, if $G$ is a simple group, then
$G$ is isomorphic to an alternating group $\mathbb{A}_n$ with
$n-l_n\geqslant 3$ (see Proposition \ref{full-degree}).
\item[$(3)$] If $D(G)=D(\mathbb{A}_n)$ and $\pi(G)=\pi(\mathbb{A}_n)$ (resp.
$D(G)=D(\mathbb{S}_n)$ and $\pi(G)=\pi(\mathbb{S}_n)$), then ${\rm
GK}(G)={\rm GK}(\mathbb{A}_n)$ (resp. ${\rm GK}(G)={\rm
GK}(\mathbb{S}_n)$). (see \cite[Lemma 2.15]{kogani})

\end{itemize}

In this section, we determine the degrees of two particular
vertices in the prime graphs of simple groups of Lie type, namely
the degree of the vertex 2 and the vertex $p$, where $p$ is the
defining characteristic. Before beginning with a general study we
want to state a consequence to Zsigmondy's theorem
\cite{zsigmondy}, given below, which will be used in the proof of
the next propositions.
\begin{lm}[Zsigmondy]\label{l1}
Let $a>1$ be a natural number. For every natural number $m$,
there exists a prime $r$ with $e(r, a)=m$ except for the
following cases: $(a, m)\in \{(2,1), (3,1), (2, 6)\}$.
\end{lm}

A prime $r$ with $e(r, a)=m$ is called a primitive prime divisor
(or a Zsigmondy prime) of $a^m-1$. By Lemma \ref{l1}, such a
prime exists except for the cases we mentioned already in the
lemma. Given a natural number $a$, we denote by $R_m(a)$ the set
of all primitive prime divisors of a number $a^m-1$. For
instance, if $a=61$ and $m=6$, then $R_6(61)=\{7, 523\}$.

We now focus our attention on simple groups of Lie type. In fact,
the order of any finite simple group of Lie type $G$ of rank $n$
over a field ${\rm GF}(q)$, $q=p^n$, is equal to
$$|G|=\frac{1}{d}q^N(q^{m_1}\pm 1)(q^{m_2}\pm 1)\cdots(q^{m_n}\pm 1),$$
(see 9.4.10 and 14.3.1 in \cite{carter}). Therefore any prime
divisor of $|G|$ distinct from the characteristic $p$ is a
primitive prime divisor of $q^m-1$, for some natural $m$. In
\cite{vv, vvc}, Vasil'ev and Vdovin found an exhaustive arithmetic
criterion of the adjacency in the prime graph of every finite
non-abelian simple group. In what follows, using the results
collected in \cite{vv, vvc}, we determine the degree of vertices
$2$ and the characteristic $p$, in the prime graph of every simple
group of Lie type.
\begin{proposition}\label{prop-Clasic-1}
Let $G=A_{n-1}(q)$ \ $(n\geqslant 2)$ defined over a field of
characteristic $p\neq 2$. Then the following hold.
\begin{itemize}
\item[{$(1)$}]
${\rm deg}(p)=\left\{ \begin{array}{lll} 0
& \mbox{if} & n=2, \\[0.3cm]
|\pi(\frac{q-1}{(3, q-1)})| & \mbox{if} & n=3, \\[0.3cm]
|\pi(G)|-|R_{n-1}(q)\cup R_{n}(q)|-1 & \mbox{if} &  n\geqslant 4.
\\
\end{array} \right.$
\item[{$(2)$}] If $n=2$, then $$\deg(2)=\left\{\begin{array}{lll}
|\pi(q-1)|-1 & \mbox{if} & 4|q-1,\\[0.2cm]
|\pi(q+1)|-1 & \mbox{if} & 4|q+1.\\[0.2cm]
\end{array} \right. $$

If $n\geqslant 3$, then
 $${\rm deg}(2)=\left\{
\begin{array}{lll} |\pi(G)|-|R_{n}(q)|-1 & \mbox{if}
& n_2<(q-1)_2, \\[0.3cm]
|\pi(G)|-|R_{n-1}(q)|-1 & \mbox{if} & (q-1)_2<n_2 \ \ \mbox{or} \ \  n_2=(q-1)_2=2, \\[0.3cm]
|\pi(G)|-|R_{n-1}(q)\cup R_{n}(q)|-1 & \mbox{if} &  n_2=(q-1)_2 \ \mbox{and} \ 4\mid {q-1}. \\
\end{array} \right.
$$ \end{itemize}
\end{proposition}
\begin{proof}
Note that $$|A_{n-1}(q)|=\frac{1}{(n,
q-1)}q^{\frac{n(n-1)}{2}}(q^2-1)(q^3-1)\cdots(q^n-1).$$

$(1)$ First, let $r\in\pi(G)\setminus \{p\}$. By Proposition 3.1
in \cite{vv}, we have $p\nsim r$ if and only if $r$ is odd and
$e(r, q)>n-2$, or equivalently, $p\nsim r$ if and only if $r$ is
odd and $e(r, q)\in\{n-1, n\}$.

As before, we recall that the vertices $2$ and $p$ are adjacent
in the prime graph ${\rm GK}(G)$, except $G=A_1(q)$. Therefore, if
$n\geqslant 4$, then since $2$ is not a Zsigmondy prime of
$q^{n-1}-1$ or $q^n-1$, we obtain
$${\rm deg}(p)=|\pi(G)|-|R_{n-1}(q)\cup R_{n}(q)|-1.$$ Now, we
consider the remaining cases, namely $n=2, 3$. Assume first that
$n=2$. In this case, we have
$$
\mu(A_1(q))=\left\{p, \frac{q-1}{2}, \frac{q+1}{2}\right\},$$ and
so ${\rm deg}(p)=0$. Assume next that $n=3$. In this case, we have
$$
\mu(A_2(q))=\left\{ \begin{array}{lll} \{q-1, \frac{p(q-1)}{3},
\frac{q^2-1}{3}, \frac{q^2+q+1}{3}\} & \mbox{if} &
d=3,\\[0.3cm]

\{p(q-1), q^2-1, q^2+q+1\} & \mbox{if} & d=1; \\
\end{array} \right.
$$
where $d=(3, q-1)$, and hence ${\rm deg}(p)=|\pi(\frac{q-1}{d})|$,
where $d=(3, q-1)$.

$(2)$ First assume that $n=2$. Again considering the spectrum of
$A_1(q)$, it is easy to see that
$$\deg(2)=\left\{\begin{array}{lll}
|\pi(q-1)|-1 & \mbox{if} & 4|q-1,\\[0.2cm]
|\pi(q+1)|-1 & \mbox{if} & 4|q+1.\\[0.2cm]
\end{array} \right. $$
Now, we may assume that $n\geqslant 3$. In this case, from
Proposition 4.1 in \cite{vv}, it is easy to see that
$${\rm deg}(2)=\left\{
\begin{array}{lll} |\pi(G)|-|R_{n}(q)|-1 & \mbox{if}
& n_2<(q-1)_2, \\[0.3cm]
|\pi(G)|-|R_{n-1}(q)|-1 & \mbox{if} & (q-1)_2<n_2 \ \ \mbox{or} \ \  n_2=(q-1)_2=2, \\[0.3cm]
|\pi(G)|-|R_{n-1}(q)\cup R_{n}(q)|-1 & \mbox{if} &  n_2=(q-1)_2 \ \mbox{and} \ 4\mid {q-1}. \\
\end{array} \right.
$$
This completes the proof.
\end{proof}
\begin{proposition}
Let $G={^2A}_{n-1}(q)$ $(n\geqslant 3)$ defined over a field of
characteristic $p\neq 2$. Then the following hold.
\begin{itemize}
\item[{$(1)$}]
If $n$ is even, then
$$
{\rm deg}(p)= \left\{
\begin{array}{lll}
|\pi(G)|-|R_{2(n-1)}(q)\cup R_{n}(q)|-1 & \mbox{if} & 4\mid n,
\\[0.3cm]
|\pi(G)|-|R_{2(n-1)}(q)|-1 & \mbox{if} &  2\parallel n. \\
\end{array}
\right.
$$
Furthermore, if $4\mid n$, then
$${\rm deg}(2)=\left\{\begin{array}{lll}
|\pi(G)|-|R_{n}(q)|-1      &   \mbox{if} & n_2<(q+1)_2,
\\[0.3cm]
|\pi(G)|-|R_{2(n-1)}(q)|-1      &   \mbox{if} &  (q+1)_2<n_2, \\[0.3cm]
|\pi(G)|-|R_n(q)\cup R_{2(n-1)}(q)|-1      &   \mbox{if} & (q+1)_2=n_2;\\
\end{array}
\right.
$$
and if  $2\parallel n$, then $ {\rm
deg}(2)=|\pi(G)|-|R_{2(n-1)}(q)|-1$.
\item[{$(2)$}]
If $n$ is odd, then
$$
{\rm deg}(p)= \left\{
\begin{array}{lll}
|\pi(G)|-|R_{n-1}(q)\cup R_{2n}(q)|-1 & \mbox{if} & 4\mid {n-1},
\\[0.3cm]
|\pi(G)|-|R_{2n}(q)|-1 & \mbox{if} & 2\parallel {n-1};  \\
\end{array}
\right.
$$
and $${\rm deg}(2)=|\pi(G)|-|R_{2n}(q)|-1. $$ \end{itemize}
\end{proposition}
\begin{proof}
First of all, we recall that $$|{^2A}_{n-1}(q)|=\left\{
\begin{array}{lll} \frac{1}{(n, q+1)}q^{\frac{n(n-1)}{2}}(q^2-1)(q^3+1)\cdots
(q^{n-1}+1)(q^n-1) & \mbox{if $n$ is even,}\\[0.3cm]
\frac{1}{(n, q+1)}q^{\frac{n(n-1)}{2}}(q^2-1)(q^3+1)\cdots
(q^{n-1}-1)(q^n+1) & \mbox{if $n$ is odd.}\\ \end{array} \right.$$

$(1)$ Let $r\in \pi(G)\setminus \{p\}$. By Proposition 3.1 in
\cite{vv}, we have $r\nsim p$ if and only if $r$ is odd and
$\nu(e(r, q))>n-2$. Assume first that $4\mid n$. On the one hand,
it follows that the Zsigmondy primes of $q^{2(n-1)}-1$ and
$q^n-1$ are not adjacent to $p$, because $\nu(2(n-1))=n-1$ and
$\nu(n)=n$. On the other hand, it is easy to see that $2\sim p$.
Therefore, we have
$${\rm deg}(p)=|\pi(G)|-|R_{2(n-1)}(q)\cup R_{n}(q)|-1.$$ Assume next that $2\parallel n$.
In this case, $p\nsim r$ if and only if $\nu(e(r,q))=n$ or $n-1$.
Clearly $\nu(e(r,q))\neq n$, otherwise $q^{n/2}-1$ must divide
the order of $G$, a contradiction. Hence we conclude that $p\nsim
r$ if and only if $\nu(e(r,q))=n-1$, in other words $p$ is not
adjacent only to the Zsigmondy primes of $q^{2(n-1)}-1$, and so
$${\rm deg}(p)=|\pi(G)|-|R_{2(n-1)}(q)|-1,$$ as desired.

The results for $\deg(2)$ follow immediately from Proposition 4.2
in \cite{vv}.

$(2)$ If $4\mid {n-1}$, then  since $\nu(2n)=n$ and
$\nu(n-1)=n-1$, from Proposition 3.1 in \cite{vv}, we conclude
that $p$ is not adjacent to the Zsigmondy primes of $q^{2n}-1$
and $q^{n-1}-1$, and also $p$ is adjacent to the rest of prime
divisors of $|G|$. Thus $${\rm deg}(p)=|\pi(G)|-|R_{n-1}(q)\cup
R_{2n}(q)|-1.$$ Moreover, if $2\parallel {n-1}$, then it is easy
to see that
$${\rm deg}(p)=|\pi(G)|-|R_{2n}(q)|-1.$$

Now, we examine the degree of vertex $2$. Since $n$ is odd,
$n_2=1<(q+1)_2$ and so from Proposition 4.2 in \cite{vv}, it
yields that the vertex $2$ is not adjacent only to the Zsigmondy
primes of $q^{2n}-1$. Hence, we get
$${\rm deg}(2)=|\pi(G)|-|R_{2n}(q)|-1, $$ which completes the proof.
\end{proof}
\begin{proposition}\label{BnCn}
Let $G=C_n(q)$ $($resp. $B_n(q)$$)$ defined over a field of
characteristic $p\neq 2$. Then the following hold.
\begin{itemize}
\item[{$(1)$}] If $n$ is even, then ${\rm deg}(p)={\rm
deg}(2)=|\pi(G)|-|R_{2n}(q)|-1$;
\item[{$(2)$}]
If $n$ is odd, then ${\rm deg}(p)=|\pi(G)|-|R_n(q)\cup
R_{2n}(q)|-1$, while
$$
{\rm deg}(2)=\left \{ \begin{array}{lll} |\pi(G)|-|R_{2n}(q)|-1 &
\mbox{if} & 4\mid{q-1}, \\[0.3cm] |\pi(G)|-|R_{n}(q)|-1 &  \mbox{if} &
4\mid{q+1}.
\end{array} \right.
$$
\end{itemize}
\end{proposition}
\begin{proof}
Recall that
$$|B_n(q)|=|C_n(q)|=q^{n^2}(q^2-1)(q^4-1)\cdots
(q^{2n}-1)/2.$$

$(1)$ First, we examine the degree of $p$. Let $r\in
\pi(G)\setminus \{p\}$, Then, by Proposition 3.1 in \cite{vv},
$p\nsim r$ if and only if $\eta(e(r, q))=n$. Since $n$ is an even
number, equivalently, it follows that $p\nsim r$ if and only if
$e(r, q)=2n$. Hence ${\rm deg}(p)=|\pi(G)|-|R_{2n}(q)|-1$, as
claimed.

Now, we examine the degree of vertex $2$. First of all, it is
evident that the vertices $2$ and $p$ are adjacent in {\em all}
classical simple groups except ${\rm GK}(A_1(q))$, especially $2$
and $p$ are adjacent in the prime graph ${\rm GK}(G)$. On the
other hand, if $r\in \pi(G)\setminus \{2, p\}$, then from
Proposition 4.3 in \cite{vv}, we see that $r\nsim 2$ if and only
if $e(r, q)=2n$, or equivalently, $r\nsim 2$ if and only if $r\in
R_{2n}(q) $. Therefore, a simple calculation shows that  ${\rm
deg}(2)=|\pi(G)|-|R_{2n}(q)|-1$, as claimed.

$(2)$ Suppose $r\in \pi(G)\setminus\{p\}$. By Proposition 3.1 in
\cite{vv}, we see that $p\nsim r$ if and only if $e(r, q)\in \{n,
2n\}$, and so ${\rm deg}(p)=|\pi(G)|-|R_n(q)\cup R_{2n}(q)|-1$.

As we mentioned in the previous paragraph, the vertices $2$ and
$p$ are adjacent in the prime graph ${\rm GK}(G)$. Assume now
that $r\in \pi(G)\setminus \{2, p\}$. Considering Proposition 4.3
in \cite{vv}, we deduce that $2\nsim r$ if and only if
$e(r,q)=(3-e(2,q))n$. But since
$$
e(r,q)=(3-e(2,q))n \Longleftrightarrow
e(r,q)=\left\{\begin{array}{lll}
 2n &
\mbox{if} & 4\mid{q-1}, \\[0.3cm] n &  \mbox{if} &
4\mid{q+1};
\end{array} \right.  \Longleftrightarrow  \left\{\begin{array}{lll} r\in R_{2n}(q) &
\mbox{if} & 4\mid{q-1}, \\[0.3cm] r\in R_n(q) &  \mbox{if} &
4\mid{q+1};
\end{array} \right.
$$
it follows that
$$
{\rm deg}(2)=\left \{ \begin{array}{lll} |\pi(G)|-|R_{2n}(q)|-1 &
\mbox{if} & 4\mid{q-1}, \\[0.3cm] |\pi(G)|-|R_{n}(q)|-1 &  \mbox{if} &
4\mid{q+1};
\end{array} \right.
$$
as claimed. Therefore, the proof is complete.
\end{proof}
\begin{proposition}\label{D_n}
Let $G=D_n(q)$ $(n\geqslant 3)$ defined over a field of
characteristic $p\neq 2$. Then the following statements hold.
\begin{itemize}
\item[{$(1)$}]
If $n$ is even, then ${\rm deg}(p)=|\pi(G)|-|R_{n-1}(q)\cup
R_{2(n-1)}(q)|-1$, and
$$
 {\rm deg}(2)= \left\{
\begin{array}{lll}
|\pi(G)|-|R_{n-1}(q)|-1    &    \mbox{if} &  4 \mid q+1, \\[0.3cm]
|\pi(G)|-|R_{2(n-1)}(q)|-1           &     \mbox{if} &  4 \mid q-1. \\
\end{array}
\right.
$$
\item[{$(2)$}] If $n$ is odd, then ${\rm deg}(p)=|\pi(G)|-|R_{n}(q)\cup
R_{2(n-1)}(q)|-1$, and
$$
{\rm deg}(2)= \left\{
\begin{array}{lll}
|\pi(G)|-|R_{n}(q)|-1    &    \mbox{if} & 2 \parallel q-1,
\\[0.3cm]
|\pi(G)|-|R_n(q)\cup R_{2(n-1)}(q)|-1     &     \mbox{if} & 4
\parallel q-1,
\\[0.3cm]
|\pi(G)|-|R_{2(n-1)}|-1 & \mbox{if} & 8 \mid q-1.
\\
\end{array}
\right.
$$
\end{itemize}
\end{proposition}
\begin{proof}
First of all, we note that $$|G|=|D_n(q)|=\frac{1}{(4,
q^n-1)}q^{n(n-1)}(q^n-1)(q^2-1)(q^4-1)\cdots(q^{2(n-1)}-1),$$ and
the fact that $2$ and $p$ are adjacent in the prime graph ${\rm
GK}(G)$.

$(1)$ Let $r\in\pi(G)\setminus \{p\}$. From Proposition 3.1 in
\cite{vv},  $r$ and $p$  are non-adjacent vertices if and only if
$\eta(e(r, q))>n-2$, or equivalently, $r\nsim p$ if and only if
$e(r, q)\in \{n-1, 2(n-1)\}$. Hence, we obtain  $${\rm
deg}(p)=|\pi(G)|-|R_{n-1}(q)\cup R_{2(n-1)}(q)|-1.$$

Now, we examine the degree of vertex $2$. Let $r\in\pi(G)\setminus
\{2, p\}$, then by Proposition 4.4 in \cite{vv}, we obtain that
$r\nsim 2$ if and only if either $e(2, q)=2$ and $r\in
R_{n-1}(q)$ or $e(2, q)=1$ and $r\in R_{2(n-1)}(q)$. Therefore,
we obtain that
$$
{\rm deg}(2)= \left\{
\begin{array}{lll}
|\pi(G)|-|R_{n-1}(q)|-1  & \mbox{if} &    4 \mid q+1,
\\[0.3cm]
|\pi(G)|-|R_{2(n-1)}(q)|-1 &     \mbox{if}     &     4 \mid q-1.\\
\end{array}
\right.
$$

$(2)$ Assume $r\in\pi(G)\setminus \{p\}$. From Proposition 3.1 in
\cite{vv}, we conclude that $r\nsim p$ if and only if $e(r, q)\in
\{n, 2(n-1)\}$. Hence, ${\rm deg}(p)=|\pi(G)|-|R_{n}(q)\cup
R_{2(n-1)}(q)|-1$.

Now, let $r\in\pi(G)\setminus \{2, p\}$. By Proposition 4.4 in
\cite{vv}, we obtain that $r\nsim 2$ if and only if either $e(r,
q)=n$ and $l\parallel {q-1}$ where $l\in \{2, 4\}$ or $e(r,
q)=2(n-1)$ and $4\mid {q-1}$. Thus, easy calculations show that
$$
 {\rm deg}(2)= \left\{
\begin{array}{lll}
|\pi(G)|-|R_{n}(q)|-1    & \mbox{if} &   2 \parallel q-1,\\[0.3cm]
|\pi(G)|-|R_n(q)\cup R_{2(n-1)}(q)|-1 & \mbox{if} &  4
\parallel q-1,\\[0.3cm]
|\pi(G)|-|R_{2(n-1)}|-1 & \mbox{if} & 8 \mid q-1. \\
\end{array}
\right.
$$
The proof of this lemma is complete.
\end{proof}
\begin{proposition}{\label{2^D_n}}
Let $G={^2D}_n(q)$ $(n\geqslant 2)$ defined over a field of
characteristic $p\neq 2$. Then the following statements hold.
\begin{itemize}
\item[{$(1)$}] If $n$ is even, then
${\rm deg}(2)=|\pi(G)|-|R_{2n}(q)|-1$, while
$${\rm deg}(p)=\left\{ \begin{array}{lll}
|\pi(G)|-|R_{2n}(q)\cup R_{2(n-1)}(q)\cup R_{n-1}(q)|-1 &
\mbox{if} &
n\geqslant 4,\\[0.3cm]
0 & \mbox{if} &  n=2. \\
\end{array} \right.
$$
\item[{$(2)$}]
If $n$ is odd, then ${\rm deg}(p)=|\pi(G)|-|R_{2n}(q)\cup
R_{2(n-1)}(q)|-1$, and also
$${\rm deg}(2)=\left\{ \begin{array}{lll}
|\pi(G)|-|R_{2n}(q)\cup R_{2(n-1)}(q)|-1 & \mbox{if} &
4\mid {q+1}, 4\parallel {q^n+1},  \\[0.3cm]
|\pi(G)|-|R_{2(n-1)}(q)|-1 & \mbox{if} & 4\mid {q+1}, 8\mid {q^n+1}, \\[0.3cm]
|\pi(G)|-|R_{2n}(q)|-1 & \mbox{if} & 4\nmid {q+1}, 2\parallel
{q^n+1}.
\\
\end{array} \right.
$$
\end{itemize}
\end{proposition}
\begin{proof}
Again, we recall that $$|G|=|{^2D}_n(q)|=\frac{1}{(4,
q^n+1)}q^{n(n-1)}(q^n+1)(q^2-1)(q^4-1)\cdots(q^{2(n-1)}-1), \ \
n\geqslant 2.$$ Note that, for $n=2$, we have ${^2D}_2(q)\cong
L_2(q^2)$.

$(1)$ First, we investigate the degree of vertex $2$. Suppose
that $r\in \pi(G)\setminus \{2, p\}$. By Proposition 4.4 in
\cite{vv}, we observe that $r\nsim 2$ if and only if $e(r, q)=2n$
and $(4, q^n+1)=(q^n+1)_2$. However, since $q$ is odd and $n$ is
even, the equality $(4, q^n+1)=(q^n+1)_2$ is always true, and so
$r\nsim 2$ if and only if $e(r, q)=2n$. On the other hand, as we
mentioned already, the vertices $2$ and $p$ are adjacent in the
prime graph of all classical simple groups (except ${\rm
GK}(A_1(q))$), especially $2\sim p$ in ${\rm GK}(G)$.
Consequently, we obtain ${\rm deg}(2)=|\pi(G)|-|R_{2n}(q)|-1$, as
required.

In what follows, we examine the degree of $p$. Let $r\in
\pi(G)\setminus \{p\}$. Then, by Proposition 3.1 in \cite{vv},
$p\nsim r$ if and only if $\eta(e(r, q))\in \{n-1, n\}$, or
equivalently, $p\nsim r$ if and only if $e(r, q)\in \{n-1,
2(n-1), 2n\}$. Hence, we get  $${\rm deg}(p)=\left\{
\begin{array}{lll}
|\pi(G)|-|R_{2n}(q)\cup R_{2(n-1)}(q)\cup R_{n-1}(q)|-1 &
\mbox{if} &
n\geqslant 4,\\[0.3cm]
0 & \mbox{if} &  n=2. \\
\end{array} \right.
$$

$(2)$ From Proposition 3.1 in \cite{vv}, we have $p\nsim r$ if and
only if $\eta(e(r, q))\in\{n-1, n\}$, or equivalently, $p\nsim r$
if and only if $e(r, q)\in \{2(n-1), 2n\}$. Therefore $${\rm
deg}(p)=|\pi(G))|-|R_{2n}(q)\cup R_{2(n-1)}(q)|-1.$$

Now, we try to finding the degree of vertex $2$. Let
$r\in\pi(G)\setminus \{2, p\}$, then by Proposition 4.4 in
\cite{vv}, we see that $r\nsim 2$ if and only if either $e(r,
q)=2(n-1)$ and $4\mid {q+1}$ or $e(r, q)=2n$ and $(4,
q^n+1)=(q^n+1)_2$. Equivalently, $r\nsim 2$ if and only if either
$e(r, q)=2(n-1)$ and $4\mid {q+1}$ or $e(r, q)=2n$ and $8\nmid
{q^n+1}$. Therefore, we obtain
$${\rm deg}(2)=\left\{ \begin{array}{lll}
|\pi(G)|-|R_{2n}(q)\cup R_{2(n-1)}(q)|-1 & \mbox{if} &
4\mid {q+1}, 4\parallel {q^n+1},  \\[0.3cm]
|\pi(G)|-|R_{2(n-1)}(q)|-1 & \mbox{if} & 4\mid {q+1}, 8\mid {q^n+1}, \\[0.3cm]
|\pi(G)|-|R_{2n}(q)|-1 & \mbox{if} & 4\nmid {q+1}, 2\parallel
{q^n+1}.
\\
\end{array} \right.
$$
as required.
\end{proof}

\begin{proposition}\label{exceptional}
Let $G$ be a finite simple exceptional group of Lie type over a
field of characteristic $p\neq 2$. Then the following statements
hold.
\begin{itemize}
\item[$(1)$] If $G=G_2(q)$, then $\deg(p)=\deg(2)=|\pi(G)|-|R_3(q)\cup
R_6(q)|-1$.
\item[$(2)$] If $G=E_6(q)$, then $$\deg(p)=|\pi(G)|-|R_8(q)\cup
R_9(q)\cup R_{12}(q)|-1,$$ and $$\deg(2)=|\pi(G)|-|R_9(q)\cup
R_{12}(q)|-1.$$
\item[$(3)$] If $G=E_7(q)$, then
$${\rm deg}(p)=|\pi(G)|-|R_{7}(q)\cup R_{9}(q)\cup R_{14}(q)\cup
R_{18}(q)|-1,$$ and  $${\rm deg}(2)=\left\{ \begin{array}{ll}
|\pi(G)|-|R_{14}(q)\cup R_{18}(q)|-1 & \mbox{if} \ \
4\mid {q-1},\\[0.3cm]
|\pi(G)|-|R_{7}(q)\cup R_{9}(q)|-1 &   \mbox{otherwise}.
\end{array} \right.
$$
\item[$(4)$] If $G={^2E}_6(q)$, then $$\deg(p)=|\pi(G)|-|R_8(q)\cup
R_{12}(q)\cup R_{18}(q)|-1,$$ and
$$\deg(2)=|\pi(G)|-|R_{12}(q)\cup R_{18}(q)|-1.$$
\item[$(5)$] If $G=E_8(q)$, then $\deg(p)=\deg(2)=|\pi(G)|-|R_{15}(q)\cup
R_{20}(q)\cup R_{24}(q)\cup R_{30}(q)|-1$.
\item[$(6)$] If $G=F_4(q)$, then $\deg(p)=|\pi(G)|-|R_8(q)\cup
R_{12}(q)|-1$ and $\deg(2)=|\pi(G)|-|R_{12}(q)|-1$.
\item[$(7)$] If $G={^3D}_4(q)$, then $\deg(p)=\deg(2)=|\pi(G)|-|R_{12}(q)|-1$.
\end{itemize}
\end{proposition}
\begin{proof}
The assertions $(1)-(7)$ follow immediately from
\cite[Propositions 3.2 and 4.5]{vv}.
\end{proof}
\begin{proposition}\label{suzuki}
Let $G$ be a finite simple Suzuki or Ree group over a field of
characteristic $p$. Then the following statements hold.
\begin{itemize}
\item[$(1)$] If $G={^2B}_2(q)$, $q=2^{2n+1}>1$, then $\deg(2)=0$.
\item[$(2)$] If $G={^2G}_2(q)$, $q=3^{2n+1}$, then $\deg(3)=1$ and
$$\deg(2)=\left\{\begin{array}{lll} |\pi(q-1)|-1 & \mbox{if} & 4\nmid
q+1,\\[0.2cm] |\pi(q^{2}-1)|-1 &  \mbox{if} & 4\mid
q+1.\\
 \end{array} \right.$$
\item[$(3)$] If $G={^2F}_4(q)$, $q=2^{2n+1}$, then
${\rm deg}(2)=|\pi(q^{4}-1)|$.
\end{itemize}
\end{proposition}
\begin{proof} The spectra of these simple groups are known and calculated as follows:
\begin{itemize}
\item $\mu({^2B}_2(q))=\{4, q-1, q-\sqrt{2q}+1,
q+\sqrt{2q}+1\}$, where $q=2^{2n+1}>2$ (see \cite{shi-1992}).
\item
$\mu({^2\!G}_2(q))=\{6, 9, q-1, (q+1)/2, q-\sqrt{3q}+1,
q+\sqrt{3q}+1\}$, where $q=3^{2n+1}$ (see \cite{brandl-shi}).
\item
$\mu({^2\!F}_4(q))=\{12, 16, 2(q+1), 4(q-1), 4(q+\sqrt{2q}+1),
4(q-\sqrt{2q}+1), q^2-1, q^2+1, q^2-q+1, (q-1)(q+\sqrt{2q}+1),
(q-1)(q-\sqrt{2q}+1), q^2+\sqrt{2q^3}+q+\sqrt{2q}+1,
q^2-\sqrt{2q^3}+q-\sqrt{2q}+1\}$, where $q=2^{2n+1}$ (see
\cite{deng-shi}).
\end{itemize}
Now, According to the above results one can easily calculate the
desired degrees in any case.
\end{proof}
\begin{proposition}\label{full-degree} Let $G$ be a finite simple group; let $\Lambda$ be the set of vertices of the
prime graph ${\rm GK}(G)$ which are joined to all other vertices.
If $|\Lambda|\geqslant 1$, then $G$ is an alternating group
$\mathbb{A}_n$ with $n-l_n\geqslant 3$. Moreover, if
$\Theta:=\{s\in \Bbb P \ | \ s\leqslant n-l_n\}$, then
$$|\Lambda|=\left\{\begin{array}{lll} |\Theta|-1 & \mbox{if} &
n-l_n=3, \\[0.3cm] |\Theta| & \mbox{if} &
n-l_n>3. \end{array} \right.$$
\end{proposition}
\begin{proof} According to the classification of finite simple groups we know
that the possibilities for $G$ are: an alternating group
$\mathbb{A}_n$, with $n\geqslant 5$; one of the 26 sporadic finite
simple groups; a simple group of Lie type. We deal with the above
cases separately:
\begin{itemize}

\item If $G$ is an alternating group $\mathbb{A}_n$ with $(n\geqslant 5)$,
then the result is obtained from Lemma 2.17 in \cite{kogani}.

\item Since $|\Lambda|\geqslant 1$, the prime graph of $G$ is connected,
while the prime graph of all sporadic simple groups is
disconnected \cite{w}. Hence, $G$ is not a sporadic simple group.

\item If $G$ is a group of Lie type, then it follows from \cite{vv,
vvc} that for every $p\in \pi(G)$ there exists $q\in \pi(G)$ such
that $p\nsim q$.
\end{itemize}
The proof is complete.
\end{proof}
\begin{lm}\label{non-OD} Let $G$ and $H$ be two finite groups such that
$|H|$ divides $|G|$ and ${\rm D}(H)={\rm D}(G)$. Assume that for
each prime $p\in \pi(|G|/|H|)$, ${\rm deg}_G(p)=|\pi(G)|-1$. Then
$G$ is not OD-characterizable.
\end{lm}
\begin{proof} The proof of the lemma is elementary by taking the
group(s) $$\tilde{G}=\tilde{H}\times H,$$ where $\tilde{H}$ is a
nilpotent group with order $|G|/|H|$. Indeed, these groups have
the same order and degree pattern as $G$.
\end{proof}

{\em Some Examples.} $(1)$ Let $n$ be a natural number such that
$n-l_n\geqslant 3$. Take $G=\mathbb{S}_n$ and $H=\mathbb{A}_n$.
Then ${\rm deg}_G(2)=|\pi(G)|-1$.  Therefore, one can easily
deduce that $$|\mathbb{Z}_2\times \mathbb{A}_n|=|G| \ \ \
\mbox{and} \ \ \ {\rm D}(\mathbb{Z}_2\times \mathbb{A}_n)={\rm
D}(G).$$ Hence $h_{\rm OD}(G)\geqslant 2$.

$(2)$ Let $G=\mathbb{A}_{10}$ and $H=J_2$. Then $|G|=2^7\cdot
3^4\cdot 5^2\cdot 7$ and $|H|=2^7\cdot 3^3\cdot 5^2\cdot 7$, and
${\rm deg}_G(3)=|\pi(G)|-1=4-1=3$. Again, we see that
$|\mathbb{Z}_3\times H|=|G|$ and ${\rm D}(\mathbb{Z}_3\times
H)={\rm D}(G)$, and so $h_{\rm OD}(G)\geqslant 2$. In \cite{mz2},
it was shown that $h_{\rm OD}(G)=2$.
\section{Characterizing Some Projective Special Linear Groups $L_4(q)$}
We now turn our attention to information about the projective
special linear groups $L_4(q)$. Recall that the order of these
groups are as follows:
$$|L_4(q)|=\frac{1}{(4, q-1)}q^6(q^2-1)(q^3-1)(q^4-1).$$
According to the results in \cite{k} and \cite{w}, the prime
graph associated with the simple group $L_4(q)$ is connected
except for $q=2, 3$ or $5$. It will be convenient to determine
the orders, spectra and degree patterns of projective special
linear groups $L_4(q)$, with $19 \leqslant q<40$, in Table 2.

{\small \begin{center} {\bf Table 2}. {\em The orders, spectra and
degree patterns of projective special linear groups $L_{4}(q)$.}
$$\begin{array}{ll} \hline
S & |S|, \ \ \mu(S), \ \ D(S) \\[0.1cm]
\hline \\[-0.2cm] L_4(19) & |S|= 2^7\cdot 3^7\cdot5^2\cdot {19}^6\cdot 127 \cdot 181 \\[0.1cm]
& \mu(S)=\{2^2\cdot 5\cdot 181, \ 3^3 \cdot 127, \ 2^2\cdot 3^2\cdot 5\cdot 19, \ 2^3\cdot 3^2\cdot 5\}\\[0.1cm]
& D(S)=(4, 4, 4, 3, 1, 2) \\[0.4cm]
L_4(23) & |S|=2^{9}\cdot 3^2\cdot 5\cdot 7\cdot 11^3\cdot {23}^6
\cdot  53\cdot 79 \\[0.1cm]  &  \mu(S)=\{2^3\cdot 3\cdot 5\cdot 53,
7\cdot 11\cdot 79, 2^3 \cdot 3\cdot 11\cdot 23, 2^4 \cdot 3\cdot 11\}\\[0.1cm]
& D(S)=(5, 5, 3, 2, 5, 3, 3, 2) \\[0.4cm]
L_4(25) & |S|=2^{10}\cdot  3^4\cdot 5^{12}\cdot 7\cdot {13}^2\cdot
31\cdot 313 \\[0.1cm]  & \mu(S)=\{13\cdot 313,
2\cdot 3^2\cdot 7\cdot 31, 2^2\cdot 3\cdot 5\cdot 13, 2^3\cdot 3\cdot 5\}\\[0.1cm]
& D(S)=(5, 5, 3, 3, 4, 3, 1) \\[0.4cm]
L_4(27) &  |S|=2^7\cdot 3^{18}\cdot 5\cdot 7^2\cdot {13}^3\cdot
73\cdot 757 \\[0.1cm]  & \mu(S)=\{2^2\cdot 5\cdot 7\cdot 73, 13\cdot 757,
2^2\cdot 3\cdot 7\cdot 13, 2^3\cdot 7\cdot 13, 9 \} \\[0.1cm]
& D(S)=(5, 3, 3, 5, 4, 3, 1) \\[0.4cm]
L_4(29) & |S|=2^7 \cdot 3^2\cdot 5^ 2\cdot 7^3\cdot 13 \cdot
{29}^6 \cdot 67\cdot 421\\[0.1cm]  & \mu(S)=\{3 \cdot 5 \cdot 421, 7\cdot 13\cdot
67, 2\cdot 3\cdot 5\cdot 7\cdot 29, 2^3\cdot 3 \cdot 5\cdot 7, 2^2\cdot 7\cdot 29\}\\[0.1cm]
& D(S)=(4, 5, 5, 6, 2, 4, 2, 2) \\[0.4cm]
L_4(31) & |S|=2 ^{13}\cdot 3^4 \cdot 5^3 \cdot 13\cdot {31}^6
\cdot 37 \cdot 331\\[0.1cm]  &
\mu(S)=\{2 ^5\cdot 13 \cdot 37,  3^2\cdot 5\cdot 331, 2^5\cdot 3\cdot 5 \cdot 31, 2^6\cdot 3\cdot 5\} \\[0.1cm]
& D(S)=(5, 4, 4, 2, 3, 2, 2) \\[0.4cm]
L_4(32) & |S|=2^{30}\cdot 3^2 \cdot 5^2 \cdot 7\cdot {11}^2\cdot
{31}^3 \cdot 41\cdot 151\\[0.1cm] &
\mu(S)=\{3\cdot 5^2 \cdot 11 \cdot 41, 7\cdot 31\cdot 151, 2\cdot 3\cdot 11\cdot 31, 3\cdot 11\cdot 31\}\\[0.1cm]
& D(S)=(3, 5, 3, 2, 5, 5, 3, 2) \\[0.4cm]
L_4(37)  & |S|=2^7 \cdot 3^7 \cdot 5 \cdot 7 \cdot {19}^2\cdot
{37}^6\cdot 67\cdot 137\\[0.1cm] &
\mu(S)=\{5\cdot 19\cdot 137,  3^3\cdot 7 \cdot 67,  2\cdot
3^2\cdot
19\cdot 37, 2^3\cdot 3^2\cdot 19, 2^2\cdot 3^2\cdot 37\} \\[0.1cm]
& D(S)=(3, 5, 2, 2, 5, 3, 2, 2) \\ \hline
\end{array}$$
\end{center}}
In the case of a generic group $G$, it is sometimes convenient to
represent the prime graph ${\rm GK}(G)$ in a compact form. By the
compact form we mean a graph whose vertices are displayed with
disjoint subsets of $\pi(G)$. Actually a vertex labeled $U$
represents the complete subgraph of ${\rm GK}(G)$ on $U$
vertices. An edge connecting $U$ and $W$ represents the set of
edges of ${\rm GK}(G)$ that connect each vertex in $U$ with each
vertex in $W$. The figures $1$-$8$ below, depict the compact form
of the prime graph of the projective special linear groups
$L_4(q)$, for $19\leqslant q\leqslant 37$.

\setlength{\unitlength}{4mm}
\begin{picture}(14,17)
\linethickness{0.3pt} %
\put(8,13){\circle*{0.5}}
\put(13,13){\circle*{0.5}}
\put(10.5,10.5){\circle*{0.5}}
\put(15.5,10.5){\circle*{0.5}}
\put(5.5,10.5){\circle*{0.5}}
\put(5.5,10.5){\line(1,1){2.5}}
\put(8,13){\line(1,0){5}}
\put(13,13){\line(1,-1){2.5}}
\put(8,13){\line(1,-1){2.5}}
\put(10.5,10.5){\line(1,1){2.5}}
\put(5,9.4){\small$127$}%
\put(7.5,13.7){\small$3$}%
\put(12.5,13.7){\small$2, 5$}%
\put(15,9.4){\small$181$}%
\put(10.3,9.4){19}%
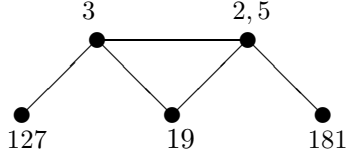
\put(6,7.5){{\small \bf Fig. 1.} \ \ ${\rm GK}(L_4(19))$.}
\put(29,13){\circle*{0.5}}
\put(32,13){\circle*{0.5}}
\put(29,10){\circle*{0.5}}
\put(32,10){\circle*{0.5}}
\put(34.2,10.8){\circle*{0.5}}
\put(26.8,10.8){\circle*{0.5}}
\put(26.8,10.8){\line(1,1){2.2}}
\put(29,13){\line(1,0){3}}
\put(32,13){\line(1,-1){2.1}}
\put(29,13){\line(1,-1){3}}
\put(29,10){\line(1,1){3}}
\put(29,10){\line(1,0){3}}
\put(29,10){\line(0,0){3}}
\put(32,10){\line(0,1){3}}
\put(25.5,9.5){\small$13, 67$}%
\put(28.8,13.7){\small$7$}%
\put(31.5,13.7){\small$3, 5$}%
\put(33.7,9.5){\small$421$}%
\put(28.8,8.8){$2$}%
\put(31.7,8.8){$29$}%
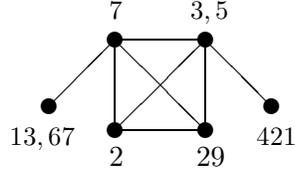
\put(26,7){{\small \bf Fig. 5.} \ \ ${\rm GK}(L_4(29))$.}
\put(8,4){\circle*{0.5}}
\put(13,4){\circle*{0.5}}
\put(10.5,1.5){\circle*{0.5}}
\put(15.5,1.5){\circle*{0.5}}
\put(5.5,1.5){\circle*{0.5}}
\put(5.5,1.5){\line(1,1){2.5}}
\put(8,4){\line(1,0){5}}
\put(13,4){\line(1,-1){2.5}}
\put(8,4){\line(1,-1){2.5}}
\put(10.5,1.5){\line(1,1){2.5}}
\put(4.5,0.3){\small$7, 79$}%
\put(7.5,4.7){\small$11$}%
\put(12.5,4.7){\small$2, 3$}%
\put(15,0.3){\small$5, 53$}%
\put(10.1,0.3){$23$}%
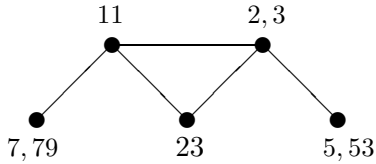
\put(6,-1.5){{\small \bf Fig. 2.} \ \ ${\rm GK}(L_4(23))$.}
\linethickness{0.3pt} %
\put(28,4){\circle*{0.5}}
\put(33,4){\circle*{0.5}}
\put(30.5,1.5){\circle*{0.5}}
\put(35.5,1.5){\circle*{0.5}}
\put(25.5,1.5){\circle*{0.5}}
\put(25.5,1.5){\line(1,1){2.5}}
\put(28,4){\line(1,0){5}}
\put(33,4){\line(1,-1){2.5}}
\put(28,4){\line(1,-1){2.5}}
\put(30.5,1.5){\line(1,1){2.5}}
\put(24.7,0.3){\small$331$}%
\put(27.3,4.7){\small$3, 5$}%
\put(32.9,4.7){\small$2$}%
\put(34.5,.3){\small$13, 37$}%
\put(30.3,.3){31}%
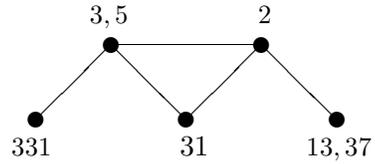
\put(26,-2){{\small \bf Fig. 6.} \ \ ${\rm GK}(L_4(31))$.}
\end{picture}

\vspace{3cm}

\setlength{\unitlength}{4mm}
\begin{picture}(0,12)
\linethickness{0.3pt} %
\put(8,15){\circle*{0.5}}
\put(13,15){\circle*{0.5}}
\put(10.5,12.5){\circle*{0.5}}
\put(15.5,12.5){\circle*{0.5}}
\put(5.5,12.5){\circle*{0.5}}
\put(5.5,12.5){\line(1,1){2.5}}
\put(8,15){\line(1,0){5}}
\put(13,15){\line(1,-1){2.5}}
\put(8,15){\line(1,-1){2.5}}
\put(10.5,12.5){\line(1,1){2.5}}
\put(5,11.4){\small$7, 31$}%
\put(7.5,15.7){\small$2, 3$}%
\put(12.5,15.7){\small$13$}%
\put(15,11.4){\small$313$}%
\put(10.3,11.4){5}%
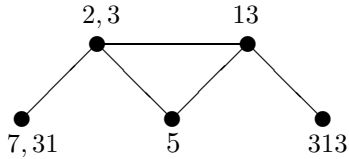
\put(6,9){{\small \bf Fig. 3.} \ \ ${\rm GK}(L_4(25))$.}
\put(28,15){\circle*{0.5}}
\put(33,15){\circle*{0.5}}
\put(30.5,12.5){\circle*{0.5}}
\put(35.5,12.5){\circle*{0.5}}
\put(25.5,12.5){\circle*{0.5}}
\put(25.5,12.5){\line(1,1){2.5}}
\put(28,15){\line(1,0){5}}
\put(33,15){\line(1,-1){2.5}}
\put(28,15){\line(1,-1){2.5}}
\put(30.5,12.5){\line(1,1){2.5}}
\put(24.5,11.4){\small$7, 151$}%
\put(27.5,15.7){\small$31$}%
\put(32.5,15.7){\small$3, 11$}%
\put(35,11.4){\small$5, 41$}%
\put(30.4,11.4){$2$}%
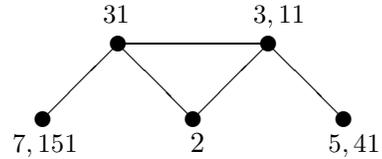
\put(26,9.5){{\small \bf Fig. 7.} \ \ ${\rm GK}(L_4(32))$.}
\put(8,6){\circle*{0.5}}
\put(13,6){\circle*{0.5}}
\put(10.5,3.5){\circle*{0.5}}
\put(15.5,3.5){\circle*{0.5}}
\put(5.5,3.5){\circle*{0.5}}
\put(5.5,3.5){\line(1,1){2.5}}
\put(8,6){\line(1,0){5}}
\put(13,6){\line(1,-1){2.5}}
\put(8,6){\line(1,-1){2.5}}
\put(10.5,3.5){\line(1,1){2.5}}
\put(5,2.3){\small$757$}%
\put(7.5,6.7){\small$13$}%
\put(12.5,6.7){\small$2, 7$}%
\put(15,2.3){\small$5, 73$}%
\put(10.3,2.3){$3$}%
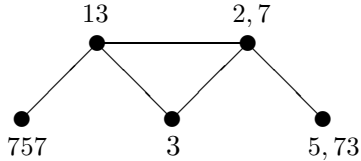
\put(6,0.5){{\small \bf Fig. 4.} \ \ ${\rm GK}(L_4(27))$.}
\put(29,6){\circle*{0.5}}
\put(32,6){\circle*{0.5}}
\put(29,3){\circle*{0.5}}
\put(32,3){\circle*{0.5}}
\put(34.2,3.8){\circle*{0.5}}
\put(26.8,3.8){\circle*{0.5}}
\put(26.8,3.8){\line(1,1){2.2}}
\put(29,6){\line(1,0){3}}
\put(32,6){\line(1,-1){2.1}}
\put(29,3){\line(1,1){3}}
\put(29,6){\line(1,-1){3}}
\put(29,3){\line(1,0){3}}
\put(29,3){\line(0,0){3}}
\put(32,3){\line(0,1){3}}
\put(26,2.7){\small $7, 67$}%
\put(28.8,6.5){\small $3$}%
\put(31.5,6.5){\small $19$}%
\put(33.5,2.7){\small $5, 137$}%
\put(28.8,1.8){$2$}%
\put(31.7,1.8){$37$}%
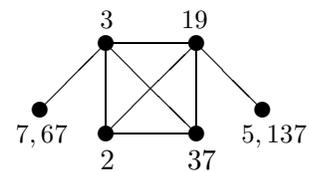
\put(26,0.5){{\small \bf Fig. 8.}  \ ${\rm GK}(L_4(37))$.}
\end{picture}
\begin{lm} \label{lem2.3} (\cite{za})
Let $P$ be a finite non-abelian simple group.
\begin{itemize}
\item[$(1)$] If $181\in \pi(P)\subseteq \{2, 3, 5, 19, 127, 181\}$, then $P$
is isomorphic to one of the following simple groups: $L_2(19^2)$,
$S_4(19)$ or $L_4(19)$.

\item[$(2)$] If $79\in \pi(P)\subseteq \{2, 3, 5, 7, 11, 23, 53, 79\}$,
then $P$ is isomorphic to one of the following simple groups:
$L_3(23)$ or $L_4(23)$.

\item[$(3)$] If $313\in \pi(P)\subseteq \{2, 3, 5, 7, 13, 31, 313\}$,
then $P$ is isomorphic to one of the following simple groups:
$L_2(5^4)$, $S_4(5^2)$, ${^2D}_4(5)$ or $L_4(5^2)$.

\item[$(4)$] If $757\in \pi(P)\subseteq \{2, 3, 5, 7, 13, 73, 757\}$,
then $P$ is isomorphic to one of the following simple groups:
$L_3(27)$ or  $L_4(27)$.

\item[$(5)$] If $421\in \pi(P)\subseteq \{2, 3, 5, 7, 13, 29, 67, 421\}$,
then $P$ is isomorphic to one of the following simple groups:
$L_2(29^2)$, $S_4(29)$ or $L_4(29)$.

\item[$(6)$] If $331\in \pi(P)\subseteq \{2, 3, 5, 13, 31, 37, 331\}$,
then $P$ is isomorphic to one of the following simple groups:
$L_3(31)$ or $L_4(31)$.

\item[$(7)$] If $151\in \pi(P)\subseteq \{2, 3, 5, 7, 11, 31, 41,
151\}$, then $P$ is isomorphic to one of the following simple
groups:  $L_3(32)$ or $L_4(32)$.

\item[$(8)$] If $137\in \pi(P)\subseteq \{2, 3, 5, 7, 19, 37, 67,
137\}$, then $P$ is isomorphic to one of the following simple
groups:  $L_2(37^2)$, $S_4(37)$ or $L_4(37)$.
\end{itemize}
\end{lm}
\begin{proof} In \cite{za}, the non-abelian finite
simple groups with prime divisors not exceeding 1000 are
determined. Given a prime $p$, we denote by $\mathcal{S}_p$ the
set of non-abelian finite simple groups $P$ such that $\max
\pi(P)=p$. Using Tables 1 and 3 in \cite{za}, we have listed the
non-abelian simple groups (except alternating ones) and their
orders in $\mathcal{S}_p$ for $p\in \{79, 137, 151, 181, 313,
331, 421, 757\}$ (see Table 5 at the end of this article). Now,
all statements $(1)$ to $(8)$ follow directly from Table 5.
\end{proof}
\begin{lm}\label{outer-prime}\cite[Lemma 2.8]{kogani} Let $P$ be a simple group and $p$ be a
prime such that $p\geqslant \max \pi(P)$. Then $p\notin \pi({\rm
Out}(P))$. \end{lm}

\begin{lm}\label{inK} Let $G$ be a finite group, with
$|G|=p_1^{m_1}p_2^{m_2}\cdots p_s^{m_s}$, where $s, m_1,
m_2,\ldots, m_s$ are positive integers and $p_1, p_2, \ldots,
p_s$ distinct primes. Let $\Delta=\{p_i\in \pi(G) \ | \ m_i=1
\}$, and for each prime $p_i\in \Delta$, let
$\Delta(p_i)=\{p_j\in \Delta \ | \ j\neq i, \ p_j\nmid p_i-1
\mbox{and} \ p_i\nmid p_j-1\}$. Let $K$ be a normal slovable
subgroup of $G$. Then there hold.
\begin{itemize}
\item[$(1)$] If $p_i\in \Delta$ divides the
order of $K$, then for each prime $p_j\in \Delta(p_i)$,  $p_j\sim
p_i$ in ${\rm GK}(G)$. In particular, $\deg_G(p_i)\geqslant
|\Delta(p_i)|$.

\item[$(2)$] If $|\Delta|=s$ and for all $i=1, 2, \ldots, s$, $|\Delta(p_i)|=s-1$, then
$G$ is a cyclic group of order $|G|$ and ${\rm D}(G)=(s-1, s-1,
\ldots, s-1)$.
\end{itemize}
\end{lm}
\begin{proof} $(1)$ If $p_j\in \pi(K)$, then $K$ contains a cyclic
subgroup of order $p_ip_j$, and the result is proved. Hence, we
may assume that $p_j\notin \pi(K)$. Let $P$ be a Sylow
$p_i$-subgroup of $K$. Then $G=KN_G(P)$ by Frattini argument, and
so $N_G(P)$ contains an element of order $p_j$, say $x$. Clearly
$P\langle x \rangle$ is a cyclic subgroup of $G$ of order
$p_ip_j$, and hence $p_i\sim p_j$ in ${\rm GK}(G)$.

$(2)$ This is a well-known result, see for example \cite[Result
1.4]{rose}. \end{proof}

\begin{theorem}[Theorem 1, \cite{vg}]\label{thm-1111}
Let $G$ be a finite group with $t(G)\geqslant 3$ and
$t(2,G)\geqslant 2$, and let $K$ be the maximal normal solvable
subgroup $K$ of $G$. Then the quotient group $G/K$ is an almost
simple group, i.e., there exists a finite non-abelian simple
group $S$ such that $S\leqslant G/K\leqslant {\rm Aut}(S)$.
\end{theorem}

\begin{lm}\label{independent-3} Let $G$ be a finite group with $|G|=|L_4(q)|$  and $D(G)=D(L_4(q))$. Then there hold.
\begin{itemize}
\item[$(1)$] $t(2,G)\geqslant 2$.
\item[$(2)$] If $t(G)\geqslant 3$, then there exists a finite non-abelian simple group $S$ such that
$S\leqslant G/K\leqslant {\rm Aut}(S)$ for the maximal normal
solvable subgroup $K$ of $G$.
\end{itemize}
\end{lm}
\begin{proof}
$(1)$ By the results summarized in Table 2 in \cite{bakbari}, it
is easy to see that $t(2,G)\geqslant 2$. Part $(2)$ follows
immediately from part $(1)$ and Theorem \ref{thm-1111}.
\end{proof}

Using the above results, we are now able to prove the following
theorem.
\begin{theorem}\label{main-2} Let $G$ be a finite group satisfies the conditions $(1) \ |G|=|L_4(q)|$ and
$(2) \ {\rm D}(G)={\rm D}(L_4(q))$, where $q\in \{19, 23, 25, 27,
29, 31, 32, 37\}$. Then $G\cong L_4(q)$.
\end{theorem}
\begin{proof}
To simplify arguments, we introduce the following notation which
will be used throughout the proof at various places.
\begin{itemize}
\item  $G\equiv $ a finite group satisfying the conditions
$(1) \ |G|=|L_4(q)|$ and $(2) \ {\rm D}(G)={\rm D}(L_4(q))$, where
$q\in \{19, 23, 25, 27, 29, 31, 32, 37\}$.
\item   $K\equiv$ the maximal normal solvable subgroup
of $G$.
\item $A\equiv$ the quotient group $G/K$.
\item $S\equiv$ the socle of $A$, that is ${\rm soc}(A)$.
\end{itemize}
We handle every case singly.

{\bf Case 1.} $q=19$. By the hypothesis $|G|=|L_4(19)|=2^7\cdot
3^7\cdot5^2\cdot {19}^6\cdot 127 \cdot 181$ and ${\rm D}(G)={\rm
D}(L_4(19))=(4, 4, 4, 3, 1, 2)$, we immediately conclude that
${\rm GK}(G)$ is a connected graph, and three possibilities may
occur (Fig. 9):

\vspace{0.95cm}

{\small  \setlength{\unitlength}{4mm}
\begin{picture}(1,1)(4,13)
\linethickness{0.3pt} %
\put(8,14){\circle*{0.5}}
\put(13,14){\circle*{0.5}}
\put(10.5,11.5){\circle*{0.5}}
\put(15.5,11.5){\circle*{0.5}}
\put(5.5,11.5){\circle*{0.5}}
\put(5.5,11.5){\line(1,1){2.5}}
\put(8,14){\line(1,0){5}}
\put(13,14){\line(1,-1){2.5}}
\put(8,14){\line(1,-1){2.5}}
\put(10.5,11.5){\line(1,1){2.5}}
\put(5,10.2){\small$127$}%
\put(7.5,14.7){\small$3$}%
\put(12.5,14.7){\small$2, 5$}%
\put(15,10.2){\small$181$}%
\put(10.1,10.2){19}%
\put(21,14){\circle*{0.5}}
\put(26,14){\circle*{0.5}}
\put(23.5,11.5){\circle*{0.5}}
\put(28.5,11.5){\circle*{0.5}}
\put(18.5,11.5){\circle*{0.5}}
\put(18.5,11.5){\line(1,1){2.5}}
\put(21,14){\line(1,0){5}}
\put(26,14){\line(1,-1){2.5}}
\put(21,14){\line(1,-1){2.5}}
\put(23.5,11.5){\line(1,1){2.5}}
\put(18,10.2){\small$127$}%
\put(20.5,14.7){\small$2$}%
\put(25.5,14.7){\small$3, 5$}%
\put(28,10.2){\small$181$}%
\put(23.1,10.2){19}%
\put(34,14){\circle*{0.5}}
\put(39,14){\circle*{0.5}}
\put(36.5,11.5){\circle*{0.5}}
\put(41.5,11.5){\circle*{0.5}}
\put(31.5,11.5){\circle*{0.5}}
\put(31.5,11.5){\line(1,1){2.5}}
\put(34,14){\line(1,0){5}}
\put(39,14){\line(1,-1){2.5}}
\put(34,14){\line(1,-1){2.5}}
\put(36.5,11.5){\line(1,1){2.5}}
\put(31,10.2){\small$127$}%
\put(33.5,14.7){\small$5$}%
\put(38.5,14.7){\small$2, 3$}%
\put(41,10.2){\small$181$}%
\put(36.1,10.2){19}%
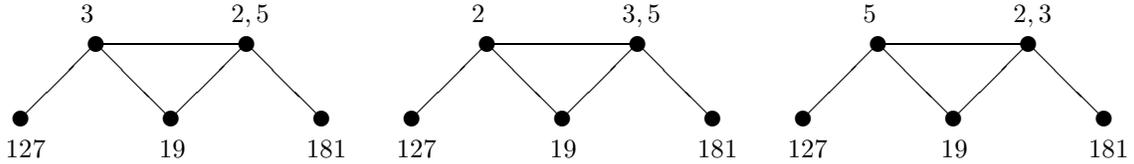
\put(13,8){{\small \bf Fig. 9.} \ \ \mbox{All possibilities for
the prime graph of} $G$.}
\end{picture}}

\vspace{2.2cm}

In all cases, we observe that $\{19, 181, 127\}$ is an
independent set. Hence, $t(G)\geqslant 3$ and it is also easy to
see that $t(2, G)\geqslant 2$. Therefore, it follows from Lemma
\ref{independent-3} that there exists a finite non-abelian simple
group $S$ such that $S\leqslant A\leqslant {\rm Aut}(S)$. We claim
that $K$ is a $\{127, 181\}'$-group. To prove this, first note
that $127\notin \pi(K)$ by Lemma \ref{inK}. Next assume that
$181\in \pi(K)$ and let $R\in {\rm Syl}_{181}(K)$. By Frattini
argument $G=KN_G(R)$. Therefore, the normalizer $N_G(R)$ contains
an element of order $127$, say $x$. Put $H:=\langle x\rangle R$.
Then $H$ is a subgroup of $G$ of order $181\cdot 127$, which is
an Abelian subgroup of $G$, and so $127\cdot 181\in \omega(G)$, a
contradiction. Now, from Lemma \ref{outer-prime}, we deduce that
$181\notin \pi({\rm Out}(S))$, which implies that $181\in
\pi(S)$. Now, by Lemma \ref{lem2.3} $(1)$, it follows that $S$ is
isomorphic to $L_4(19)$, and since $L_4(19)\leqslant G/K
\leqslant {\rm Aut}(L_4(19))$ and $|G|=|L_4(19)|$,  $K=1$ and
$G\cong L_4(19)$.

{\bf Case 2.} $q=23$. In this case, we have
$|G|=|L_4(23)|=2^9\cdot 3^2\cdot 5\cdot 7\cdot {11}^3\cdot
{23}^6\cdot 53\cdot 79$, and ${\rm D}(G)={\rm D}(L_4(23))=(5, 5,
3, 2, 5, 3, 3, 2)$. By Lemma \ref{inK} $(1)$, $K$ is a $\{7,
79\}'$-group because $\deg_G(79)<|\Delta(79)|=3$ and
$\deg_G(7)<|\Delta(7)|=3$. In addition, we have $S=P_1\times
\cdots \times P_m$, where each $P_i$ is a non-abelian simple
group, and $S\leqslant A\leqslant {\rm Aut}(S)$. We claim that
$m=1$ and $S\cong L_4(23)$. Suppose $m\geqslant 2$. It is clear
that $79$ does not divide the order of $S$, otherwise ${\rm
deg}(79)\geqslant 3$, which is a contradiction. Note that
$79\in\pi(A)\subseteq\pi({\rm Aut}(S))$, and so $79$ divides
$|{\rm Out}(S)|$. However, since
$${\rm Out}(S)={\rm Out}(S_1)\times \cdots\times {\rm
Out}(S_r),$$ where the groups $S_j$ are direct products of
isomorphic $P_i$'s such that $S\cong S_1\times \cdots \times S_r$,
for some $j$, $79$ divides $|{\rm Out}(S_j)|=|{\rm
Aut}(S_j)|/|S_j|$ where $S_j$ is a direct product of $t$
isomorphic simple groups $P_i$. On the one hand, since
$P_i\in{\cal{S}}_{53}$, the order of $S_j$ is not divisible by
$79$, and so $79 | \ |{\rm Aut}(S_j)|$. On the other hand, since
$P_i\in{\cal{S}}_{53}$, it follows that $|{\rm Aut}(P_i)|$ is not
divisible by $79$, and since
$$|{\rm Aut}(S_j)|=|{\rm Aut}(P_i)|^t\cdot t!,$$ we conclude that
$t\geqslant 79$. But then $4^{79}$ must divide $|G|$, which is a
contradiction. Therefore $m=1$ and $S\cong P_1$ is a simple group.

Clearly $S\in {\cal{S}}_{79}$. It follows from Lemma
\ref{outer-prime} that $79\notin \pi({\rm Out}(S))$, while $79\in
{\rm Aut}(S)$. Hence $79\in \pi(S)$, and Lemma \ref{lem2.3} $(2)$
shows that $S$ is isomorphic to $L_3(23)$ or $L_4(23)$. If
$S\cong L_3(23)$, then $|S|=2^5\cdot 3\cdot 7\cdot 11^2\cdot
23^3\cdot 79$ and since $|{\rm Out}(S)|=4$, $|K|$ is divisible by
$5\cdot 53$. By Lemma \ref{inK} and condition $(2)$, $\{5, 7, 23,
79\}$ is a connected component of ${\rm GK}(G)$, which
contradicts the assumption $(2)$. Therefore, $S=L_4(23)$ and
$L_4(23)\leqslant\frac{G}{K}\leqslant{\rm Aut}(L_4(23))$. Since
$|G|=|L_4(23)|$, $|K|=1$ and $G$ is isomorphic to $L_4(23)$.

{\bf Case 3.} $q=25$. Here, $G$ is a finite group with
$|G|=|L_4(25)|=2^{10}\cdot 3^4\cdot 5^{12}\cdot 7\cdot
{13}^2\cdot 31\cdot 313$ and ${\rm D}(G)={\rm D}(L_4(25))=(5, 5,
3, 3, 4, 3, 1)$. Since $|\Delta(313)|=2$ and $\deg_G(313)=1$,
 $|K|$ is not divisible by $313$ by Lemma \ref{inK}.
Moreover, as previous case, $S=P_1\times \cdots \times P_m$, where
each $P_i$ is a non-abelian simple group, and $S\leqslant
A\leqslant {\rm Aut}(S)$. Again we claim that $m=1$ and $S\cong
L_4(25)$. Suppose $m\geqslant 2$. It is clear that $313$ does not
divide $|S|$, otherwise ${\rm deg}(313)\geqslant 2$, which is a
contradiction. On the other hand, $313$ divides $|{\rm Out}(S)|$
because $313\in\pi(A)\subseteq\pi({\rm Aut}(S))$. In a similar
way as in the previous case, it follows that $2^{626}$ divides
$|G|$, which is a contradiction. Therefore, $m=1$ and $S\cong
P_1$. In the sequel, we show that $P_1\cong L_4(25)$.

Certainly $S\in {\cal{S}}_{313}$. Moreover, by Lemma
\ref{outer-prime}, $313\notin \pi({\rm Out}(S))$, while $313\in
{\rm Aut}(S)$. Hence $313\in \pi(S)$, and Lemma \ref{lem2.3} $(3)$
shows that $S$ is isomorphic to one of the following simple
groups: $L_2(5^4)$, $S_4(5^2)$, ${^2D}_4(5)$ or $L_4(5^2)$. We
recall that:
$$\begin{array}{lll} |L_2(5^4)|=2^{4}\cdot 3\cdot 5^4\cdot
13\cdot 313, &  \ \   & |S_4(5^2)|=2^{8}\cdot 3^2\cdot 5^8\cdot
{13}^2\cdot
313,\\[0.2cm]
|^2D_4(5)|=2^{10}\cdot 3^4\cdot 5^{12}\cdot 7\cdot 13\cdot
31\cdot {313}, & \ \ &  |L_4(5^2)|=2^{10}\cdot 3^4\cdot
5^{12}\cdot 7\cdot {13}^2\cdot 31\cdot 313.\\
\end{array} $$
If $S$ is isomorphic to $L_2(5^4)$ or $S_4(5^2)$, then $7, 31\in
\pi(K)$ and hence by Lemma \ref{inK}, we obtain
$\deg(313)\geqslant 2$, which is a contradiction.

If $S\cong {^2D}_4(5)$, then $13\sim p$ in ${\rm GK}(S)$, and so
in ${\rm GK}(G)$, for each $p\in \{2, 3, 5\}$ (see Propositions
3.1 $(5)$ and 4.4 in \cite{vv} and Proposition 2.5 in
\cite{vvc}), hence $\deg_S(13)=3\leqslant \deg_G(13)$. On the
other hand, since $|{\rm Out}(S)|=4$, $K$ is a normal subgroup of
$G$ of order 13. Let $Q$ be a $q$-Sylow subgroup of $G$ where
$q\in \{7, 31\}$. Then $KQ$ is a cyclic subgroup of $G$ of order
$13q$, and so $13\sim q$ in ${\rm GK}(G)$,  which forces
$\deg_G(13)\geqslant 5$, this cannot be the case.

Thus $S\cong L_4(25)$ and
$L_4(25)\leqslant\frac{G}{K}\leqslant{\rm Aut}(L_4(25))$. It
follows now from $|G|=|L_4(25)|$ that $|K|=1$ and $G$ is
isomorphic to $L_4(25)$.

{\bf Case 4.} $q=27$. In this case, we observe that
$$|G|=|L_4(27)|=2^7\cdot 3^{18}\cdot 5\cdot 7^2\cdot {13}^3\cdot
73\cdot 757 \ \mbox{and} \ {\rm D}(G)={\rm D}(L_4(27))=(5, 3, 3,
5, 4, 3, 1).$$  Since $|\Delta(757)|=2$ and $\deg_G(757)=1$,
 $757\notin \pi(K)$ by Lemma \ref{inK}.
In a similar way as before it can be shown that $S$ is a
non-abelian simple group and $S\leqslant A\leqslant {\rm Aut}(S)$.
We will show that $S\cong L_4(27)$. First of all, it is easily
seen that $757\in \pi(S)$. Thus by Lemma \ref{lem2.3} $(4)$, $S$
is isomorphic to one of the groups: $L_3(27)$ or $L_4(27)$. If
$S\cong L_3(27)$, then $|S|=2^{4}\cdot 3^9\cdot 7\cdot
{13}^2\cdot 757$ and since $|{\rm Out}(S)|=6$, $|K|$ is divisible
by $5\cdot 73$. By Lemma \ref{inK}, $5\sim 757$ and $73\sim 757$
in ${\rm GK}(G)$, and so $\deg_G(757)\geqslant 2$, which is not
the case. Thus $S\cong L_4(27)$, as required. Finally from
$L_4(27)\leqslant \frac{G}{K}\leqslant{\rm Aut}(L_4(27))$ and
$|G|=|L_4(27)|$, we conclude that $|K|=1$ and $G$ is isomorphic
to $L_4(27)$.

{\bf Case 5.} $q=29$. Here, $G$ is a finite group with
$$|G|=|L_4(29)|=2^7\cdot 3^2\cdot 5^2\cdot 7^3\cdot {13}\cdot
{29}^6\cdot 67\cdot 421 \ \mbox{and}  \ {\rm D}(G)={\rm
D}(L_4(29))=(4, 5, 5, 6, 2, 4, 2, 2).$$ First, we claim that the
order of $K$ is not divisible by $421$. Assume the contrary.
Then, in view of Lemma \ref{inK}, $421\sim 13$ and $421\sim 67$
in ${\rm GK}(G)$. Moreover, we have $\deg_G(7)=6$ and this forces
$7\sim q$ for $q\in \pi(G)\setminus \{7, 421\}$. Therefore, for
each $p\in \pi(G)\setminus \{7, 13, 67, 421\}$, it follows that
$\deg_G(p)\leqslant |\pi(G)|-4=4$, which is a contradiction.

Using a similar argument as previous case, one can see that $S$ is
a non-abelian simple group, and $S\leqslant A\leqslant {\rm
Aut}(S)$. In the sequel, we will show that $S\cong L_4(29)$.
Indeed, by what observed above and Lemma \ref{outer-prime},
$421\notin \pi(K)\cup \pi({\rm Out}(S))$, thus $421\in \pi(S)$.
Now, from Lemma \ref{lem2.3} $(5)$, $S$ is isomorphic to one of
the groups: $L_2(29^2)$, $S_4(29)$ or $L_4(29)$. Note that
$|L_2({29}^2)|=2^{3}\cdot 3\cdot 5\cdot 7\cdot {29}^2\cdot 421$
and $|S_4(29)|=2^{6}\cdot 3^2\cdot 5^2\cdot 7^2\cdot {29}^4\cdot
421$. If $S$ is isomorphic to $L_2(29^2)$ or $S_4(29)$, then $13,
67\in \pi(K)$.  In view of Lemma \ref{inK}, we observe that
$\{13, 67, 421\}$ is a clique in ${\rm GK}(G)$, and this forces
$\deg_G(3)\leqslant 4$, a contradiction. Therefore, $S$ is
isomorphic to $L_4(29)$, and so $L_4(29)\leqslant
\frac{G}{K}\leqslant {\rm Aut}(L_4(29))$. Since $|G|=|L_4(29)|$,
we obtain $|K|=1$ and $G$ is isomorphic to $L_4(29)$.

{\bf Case 6.} $q=31$. In this case, we have
$$|G|=|L_4(31)|=2^{13}\cdot 3^4\cdot 5^3\cdot {13}\cdot
{31}^6\cdot 37\cdot 331  \ \mbox{and} \ {\rm D}(G)={\rm
D}(L_4(31))=(5, 4, 4, 2, 3, 2, 2).$$ In a similar way as in the
proof of Case 5, we see that $331\notin \pi(K)$ and $S$ is a
non-abelian simple group such that $S\leqslant A\leqslant {\rm
Aut}(S)$. In addition, $331\in \pi(S)$ and by Lemma \ref{lem2.3}
$(6)$, $S$ is isomorphic to one of the groups: $L_3(31)$ or
$L_4(31)$. If $S$ is isomorphic to $L_3(31)$, then
$|L_3(31)|=2^{7}\cdot 3^2\cdot 5^2\cdot {31}^3\cdot 331$, and we
conclude that $13, 37\in \pi(K)$. Now it follows from Lemma
\ref{inK} that $\{13, 37, 331\}$ is a clique in ${\rm GK}(G)$,
and this forces $\deg_G(2)\leqslant 3$, a contradiction. Finally,
$S\cong L_4(31)$, and hence
$L_4(31)\leqslant\frac{G}{K}\leqslant{\rm Aut}(L_4(31))$. Since
$|G|=|L_4(31)|$, we obtain $|K|=1$ and $G$ is isomorphic to
$L_4(31)$.

{\bf Case 7.} $q=32$. Here, we deal with a finite group $G$ with
$$|G|=|L_4(32)|=2^{30}\cdot 3^2\cdot 5^2\cdot 7\cdot {11}^2\cdot
{31}^3\cdot 41\cdot 151  \ \mbox{and}  \ {\rm D}(G)={\rm
D}(L_4(32))=(3, 5, 3, 2, 5, 5, 3, 2).$$ We claim that the order of
$K$ is coprime to $151$. Suppose not. Then, it follows from Lemma
\ref{inK} that $151\sim 41$ and $151\sim 7$. If the order of $K$
is divisible by 11, then we consider a $\{11, 151 \}$-Hall
subgroup of $K$, say $L$. Clearly $|L|=11^i\cdot 151$, where
$i=1$ or $2$. In both cases the group $L$ is abelian and so
$11\sim 151$ in ${\rm GK}(G)$. Thus $\deg_G(151)\geqslant 3$,
which is a contradiction. Therefore we may assume that $11\notin
\pi(K)$. Let $P\in {\rm Syl}_{151}(G)$. Then $G=KN_G(P)$ by
Frattini argument, and so $N_G(P)$ contains an element of order
$11$, say $x$. Clearly $P\langle x \rangle$ is a cyclic subgroup
of $G$ of order $11\cdot 151$, and hence $11\sim 151$ in ${\rm
GK}(G)$, again a contradiction. This completes the proof of our
claim.

Finally, as previous cases, we can show that $S$ is a non-abelian
simple group, and $S\leqslant A\leqslant {\rm Aut}(S)$. We claim
that $S\cong L_4(32)$. Since $151\notin \pi(K)\cup \pi({\rm
Out}(S))$,  $151\in \pi(S)$ and by Lemma \ref{lem2.3} $(7)$, $S$
is isomorphic to one of the groups: $L_3(32)$ or $L_4(32)$. If
$S$ is isomorphic to $L_3(32)$, then $|S|=2^{15}\cdot 3\cdot
7\cdot 11\cdot 31^2\cdot 151$, and since $|{\rm Out}(S)|=10$, we
conclude that $|K|=2^\alpha\cdot 3\cdot 5^\beta\cdot 11\cdot
31\cdot 41$, where $14\leqslant \alpha\leqslant 15$ and
$1\leqslant \beta\leqslant 2$. One the one hand, we consider a
$\{41, p\}$-Hall subgroup of $K$, where $p\in \{3, 11, 31\}$.
Such subgroups are clearly cyclic, which implies that $41\sim p$
in ${\rm GK}(K)$ and so in ${\rm GK}(G)$. On the other hand, from
Lemma \ref{inK}, it follows that $7\sim 41\sim 151$ in ${\rm
GK}(G)$. Thus $\deg_G(41)\geqslant 5$, which is a contradiction.
Therefore $S\cong L_4(32)$ and
$L_4(32)\leqslant\frac{G}{K}\leqslant{\rm Aut}(L_4(32)).$
Moreover, since $|G|=|L_4(32)|$, $|K|=1$ and $G$ is clearly
isomorphic to $L_4(32)$.

{\bf Case 8.} $q=37$. In this case, we have $$|G|=|L_4(37)|=2^7
\cdot 3^7 \cdot 5 \cdot 7 \cdot {19}^2\cdot {37}^6\cdot 67\cdot
137  \ \  \mbox{and} \ \   D(G)=D(L_4(37))=(3, 5, 2, 2, 5, 3, 2,
2).$$ Since $|\Delta(137)|=3$ and $\deg_G(137)=2$, the order of
$K$ is not divisible by $137$ by Lemma \ref{inK}. By a similar
argument, as before, it follows that $S$ is a non-abelian simple
group, and $S\leqslant A\leqslant {\rm Aut}(S)$. We next show that
$S\cong L_4(47)$.

Since $137\notin \pi(K)$ and $137\notin \pi({\rm Out}(S))$ by
Lemma \ref{outer-prime}, we conclude that $137\in \pi(S)$. Finally
$S$ is isomorphic to one of the groups described in Lemma
\ref{lem2.3} $(8)$: $L_2(37^2)$, $S_4(37)$ or $L_4(37)$. If $S$
is isomorphic to $L_2(37^2)$ or $S_4(37)$, then $|K|$ is divisible
by $67$, and we get $67\sim p$ where $p\in\{5, 7, 137\}$, hence
$\deg_G(67)\geqslant 3$, a contradiction. Thus $S\cong L_4(37)$,
and so $L_4(47)\leqslant\frac{G}{K}\leqslant{\rm Aut}(L_4(47))$.
Moreover, since $|G|=|L_4(47)|$, $|K|=1$ and hence $G$ is
isomorphic to $L_4(47)$.
\end{proof}
\section{Appendix}
In a series of articles, it was shown that many finite simple
groups are OD-characterizable or 2-fold OD-characterizable. Table
3 lists finite simple groups which are currently known to be
$k$-fold OD-characterizable for $k\in \{1, 2\}$. Until recently,
no examples of simple groups $P$ with $h_{\rm OD}(P)\geqslant 3$
were known. Therefore, we posed the following question:
\begin{problem} Is there a non-abelian simple group $P$ with $h_{\rm
OD}(P)\geqslant 3$?
\end{problem}

\begin{center}
{\bf Table 3}. Some non-abelian simple groups $S$ with $h_{\rm
OD}(S)=1$ or $2$.\\[0.4cm]
$\begin{array}{l|l|c|l} \hline S & {\rm Conditions \ on} \ S&
h_{\rm OD}(S) & {\rm Refs.} \\ \hline
 \mathbb{A}_n & \ n=p, p+1, p+2 \ (p \ {\rm a \ prime})& 1 &  \cite{mz1}, \cite{mzd}    \\
 & \ 5\leqslant n\leqslant 100, n\neq 10   & 1 & \cite{hm}, \cite{kogani}, \cite{banoo-alireza},  \\
 & & & \cite{mz3}, \cite{zs-new2} \\
& \ n=106, \ 112 & 1 &      \cite{yan-chen}       \\
& \ n=10 & 2 &      \cite{mz2}       \\[0.2cm]
L_2(q) &  q\neq 2, 3& 1 &    \cite{mz1}, \cite{mzd},\\
& & &  \cite{zshi} \\[0.1cm]
L_3(q) &  \ |\pi(\frac{q^2+q+1}{d})|=1, \ d=(3, q-1) & 1 &   \cite{mzd} \\[0.2cm]
U_3(q) &  \ |\pi(\frac{q^2-q+1}{d})|=1, \ d=(3, q+1), q>5 & 1 &   \cite{mzd} \\[0.2cm]
L_3(9) & & 1 & \cite{zs-new4}\\[0.1cm]
U_3(5) &   & 1 &   \cite{zs-new5} \\[0.1cm]
L_4(q) &  \ q\leqslant 17  & 1 &   \cite{bakbari, amr} \\[0.1cm]
U_4(7) &   & 1 &   \cite{amr} \\[0.1cm]
L_n(2) & \ n=p \ {\rm or} \ p+1, \ {\rm for \ which} \ 2^p-1 \ {\rm is \ a \ prime} & 1 & \cite{amr} \\[0.2cm]
L_n(2) & \ n=9, 10, 11  & 1 &   \cite{khoshravi}, \cite{R-M} \\[0.1cm]
U_6(2) & & 1 & \cite{LShi} \\[0.1cm]

R(q) & \ |\pi(q\pm \sqrt{3q}+1)|=1, \ q=3^{2m+1}, \ m\geqslant 1 & 1 & \cite{mzd} \\[0.2cm]
{\rm Sz} (q) & \ q=2^{2n+1}\geqslant 8& 1 &   \cite{mz1}, \cite{mzd} \\[0.2cm]
B_m(q), C_m(q) &  m=2^f\geqslant 4,  \
|\pi\big((q^m+1)/2\big)|=1, \
 & 2 & \cite{akbarim}\\[0.2cm]
B_2(q)\cong C_2(q) &  \ |\pi\big((q^2+1)/2\big)|=1, \ q\neq
3 & 1 & \cite{akbarim}\\[0.2cm]
B_m(q)\cong C_m(q) &  m=2^f\geqslant 2, \ 2|q, \
|\pi\big(q^m+1\big)|=1, \ (m, q)\neq (2, 2) & 1 &
\cite{akbarim}\\[0.2cm]
B_p(3), C_p(3) &  |\pi\big((3^p-1)/2\big)|=1, \  p \ {\rm is \ an
\ odd \
prime}  & 2 & \cite{akbarim}, \cite{mzd}\\[0.2cm]
B_3(5), C_3(5) & & 2 & \cite{akbarim} \\[0.2cm]
C_3(4) & & 1 & \cite{moghadam} \\[0.1cm]
S &  \ \mbox{A sporadic simple group} & 1 & \cite{mzd} \\[0.1cm]
S &  \ \mbox{A simple group with} \ |\pi(S)|=4, \ \ S\neq \mathbb{A}_{10} & 1 & \cite{zs} \\[0.1cm]
S &  \  \mbox{A simple group with} \ |S|\leqslant 10^8, \ \ S\neq \mathbb{A}_{10}, \ U_4(2) & 1 & \cite{ls} \\[0.1cm]
S &  \  \mbox{A simple $C_{2,2}$- group} & 1 & \cite{mz1}
\end{array}$
\end{center}
\footnotetext{In Table 3, $q$ is a power of a prime number.}

Although we have not found a simple group which is $k$-fold
OD-characterizable for $k\geqslant 3$, but among non-simple
groups, there are many groups which are $k$-fold
OD-characterizable for $k\geqslant 3$. As an easy example, if $P$
is a $p$-group of order $p^n$, then $h_{\rm OD}(P)=\nu(p^n)$,
where $\nu(m)$ signifies the number of non-isomorphic groups of
order $m$. Table 4 lists finite non-solvable groups which are
currently known to be OD-characterizable or $k$-fold
OD-characterizable with $k\geqslant 2$.

\begin{center}
{\bf Table 4}. Some non-solvable groups $G$ with known $h_{\rm
OD}(G)$.\\[0.2cm]
$\begin{array}{l|l|c|l} \hline G & {\rm Conditions \ on} \ G &
h_{\rm OD}(G) & {\rm Refs.} \\ \hline
{\rm Aut}(M) & M \ \mbox{is a sporadic group}  \neq  J_2, M^cL    & 1 & \cite{mz1} \\[0.1cm]
\mathbb{S}_n & n=p, \ p+1 \ (p\geqslant 5 \ \mbox{is a prime})& 1 &  \cite{mz1}    \\[0.1cm]
U_3(5): 2 & & 1 & \cite{zs-new5} \\[0.1cm]
U_6(2): 2 & &1& \cite{LShi}\\[0.1cm]
M &  M\in \mathcal{C}_1 & 2 &      \cite{mz2}       \\[0.1cm]
M & M\in \mathcal{C}_2 & 8 &      \cite{mz2}       \\[0.1cm]
M & M\in \mathcal{C}_3  & 3 & \cite{hm, kogani, banoo-alireza, mz3, yan-chen} \\[0.1cm]
M & M\in \mathcal{C}_4     & 2 & \cite{mz2} \\[0.1cm]
M & M\in \mathcal{C}_5    & 3 & \cite{mz2} \\[0.1cm]
M & M\in \mathcal{C}_6  & 6 &\cite{banoo-alireza} \\[0.1cm]
M & M\in \mathcal{C}_7 & 1 &  \cite{zs-new1} \\[0.1cm]
M & M\in \mathcal{C}_8  & 9 &\cite{zs-new1} \\[0.1cm]
M & M\in \mathcal{C}_{9} & 3 &  \cite{zs-new5} \\[0.1cm]
M & M\in \mathcal{C}_{10}  & 6 &\cite{zs-new5}\\[0.1cm]
M & M\in \mathcal{C}_{11}  & 3 &\cite{LShi}\\[0.1cm]
M & M\in \mathcal{C}_{12}  & 5 &\cite{LShi}\\[0.1cm]
M & M\in \mathcal{C}_{13}  & 1 &\cite{y-chen-w}\\[0.1cm]
M & M\in \mathcal{C}_{14}  & 1 &\cite{R-M}
\end{array}$
\end{center}
\begin{tabular}{lll}
$\mathcal{C}_1$ & $\!\!\!\!=$  & $ \!\!\!\! \{\mathbb{A}_{10}, J_2\times \mathbb{Z}_{3} \}$\\[0.1cm]
$\mathcal{C}_2$ & $\!\!\!\!=$ & $\!\!\!\! \{ \mathbb{S}_{10},  \
\mathbb{Z}_{2}\times \mathbb{A}_{10}, \ \mathbb{Z}_2\cdot
\mathbb{A}_{10}, \ \mathbb{Z}_6\times J_2,  \ \mathbb{S}_3 \times
J_2, \ \mathbb{Z}_3\times (\mathbb{Z}_2\cdot J_2),
$ \\[0.1cm]
& & $(\mathbb{Z}_3\times J_2)\cdot \mathbb{Z}_2, \ \mathbb{Z}_3\times {\rm Aut}(J_2)\}.$\\[0.1cm]
$ \mathcal{C}_3$ & $\!\!\!\!=$ & $\!\!\!\! \{\mathbb{S}_{n}, \
\mathbb{Z}_{2}\cdot \mathbb{A}_{n}, \ \mathbb{Z}_{2}\times
\mathbb{A}_{n}\}$, \
where $9\leqslant n\leqslant 100$ with $n\neq 10, p, p+1$ ($p$ a prime)\\[0.1cm]
& & or $n=106, \ 112$.\\[0.1cm]
$\mathcal{C}_4$ & $\!\!\!\!=$ & $\!\!\!\! \{ {\rm Aut}(M^cL), \ \mathbb{Z}_2\times M^cL\}$.\\[0.1cm]
$\mathcal{C}_5$ & $\!\!\!\!=$ & $\!\!\!\! \{{\rm Aut}(J_2), \
\mathbb{Z}_2\times J_2, \
\mathbb{Z}_2\cdot J_2\}.$\\[0.1cm]
$\mathcal{C}_6$ & $\!\!\!\!=$ & $\!\!\!\! \{{\rm Aut}(S_6(3)), \
\mathbb{Z}_2\times S_6(3), \  \mathbb{Z}_2\cdot S_6(3),  \
\mathbb{Z}_2\times O_7(3), \
\mathbb{Z}_2\cdot O_7(3)$, \ ${\rm Aut}(O_7(3))\}$.\\[0.1cm]
$\mathcal{C}_7$ & $\!\!\!\!=$ & $\!\!\!\! \{L_2(49):2_1, \
L_2(49):2_2, \
L_2(49):2_3\}$.\\[0.1cm]
$\mathcal{C}_8$ & $\!\!\!\!=$ & $\!\!\!\! \{L\cdot 2^2,  \
\mathbb{Z}_{2}\times (L: 2_1), \ \mathbb{Z}_2\times (L: 2_2), \
\mathbb{Z}_2\times (L: 2_3),  \ \mathbb{Z}_2\cdot(L: 2_1), $
\\[0.1cm]
& & $\mathbb{Z}_2\cdot(L: 2_2), \ \mathbb{Z}_2\cdot(L: 2_3), \
\mathbb{Z}_4\times L, \ (\mathbb{Z}_2\times \mathbb{Z}_2)\times
L\}$, \ where $L=L_2(49)$.\\[0.1cm]
$\mathcal{C}_{9}$ & $\!\!\!\!=$  & $ \!\!\!\! \{U_3(5):3, \
\mathbb{Z}_{3} \times U_3(5),
\ \mathbb{Z}_{3}\cdot U_3(5) \}$\\[0.1cm]
$\mathcal{C}_{10}$ & $\!\!\!\!=$  & $ \!\!\!\! \{L:\mathbb{S}_3, \
\mathbb{Z}_{2}\cdot(L:3), \ \mathbb{Z}_{3}\times (L:2), \
\mathbb{Z}_{3}\cdot(L:2), \ (\mathbb{Z}_{2}\times L): 2$, \
$(\mathbb{Z}_{3}\cdot L): 2\},$
\\[0.1cm]
& & where
$L=U_3(5)$.\\[0.1cm]
$\mathcal{C}_{11}$ & $\!\!\!\!=$  & $ \!\!\!\! \{U_6(2): 3, \
\mathbb{Z}_{3}\times U_6(2), \ \mathbb{Z}_{3}\cdot U_6(2)\}$.\\[0.1cm]
$\mathcal{C}_{12}$ & $\!\!\!\!=$  & $ \!\!\!\! \{L\cdot
\mathbb{S}_3, \ \mathbb{Z}_{3}\times (L:2), \ \mathbb{Z}_{3}\cdot
(L:2), \ (\mathbb{Z}_{3}\times L):2, \ (\mathbb{Z}_{3}\cdot
L):2\}$, where
$L=U_6(2)$.\\[0.1cm]
$\mathcal{C}_{13}$ & $\!\!\!\!=$  & $ \!\!\!\! \{{\rm
Aut}(O^+_{10}(2), \  {\rm Aut}(O^-_{10}(2)\}$,
\\[0.1cm]
$\mathcal{C}_{14}$ & $\!\!\!\!=$  & $ \!\!\!\! \{{\rm
Aut}(L_{p}(2)), \ {\rm Aut}(L_{p+1}(2))\}$, where $2^p-1$ is a
Mersenne prime.
\\[0.1cm]
\end{tabular}
\begin{center}
{\bf Table 5}. {\em The non-abelian simple groups $S\in
{\cal{S}}_{p}$ except alternating ones.}

 \vspace{.5cm}

\begin{tabular}{lll} \hline
$p$ & $S$ & $|S|$ \\[0.1cm]
\hline
79 & $L_3(23)$ & $2^5\cdot3\cdot7\cdot11^2\cdot23^3\cdot79$\\[0.1cm]
& $L_4(23)$ & $2^9\cdot3^2\cdot5\cdot7\cdot11^3\cdot23^6\cdot53\cdot79$\\[0.1cm]
& $O_7(23)$ & $2^{12}\cdot3^4\cdot5\cdot7\cdot11^3\cdot13^2\cdot23^9\cdot53\cdot79$ \\[0.1cm]
& $S_6(23)$ & $2^{12}\cdot3^4\cdot5\cdot7\cdot11^3\cdot13^2\cdot23^9\cdot53\cdot79$  \\[0.1cm]
&$O_8^+(23)$ &
$2^{16}\cdot3^5\cdot5^2\cdot7\cdot11^4\cdot13^2\cdot23^{12}\cdot53^2\cdot79$  \\[0.1cm]
&$G_2(23)$ & $2^8\cdot3^3\cdot7\cdot11^2\cdot13^2\cdot23^6\cdot79$\\[0.1cm]
&$L_3(23^2)$ & $2^9\cdot3^2\cdot5\cdot7\cdot11^2\cdot13^2\cdot23^6\cdot 53\cdot79$  \\[0.1cm]
&$L_2(23^3)$ & $2^3\cdot3^2\cdot7\cdot11\cdot13^2\cdot 23^3\cdot79$ \\[0.1cm]
&$L_2(79)$ & $2^4\cdot3\cdot5\cdot13\cdot79$ \\[0.1cm]
&$L_3(79)$ & $2^6\cdot3^2\cdot5\cdot7^2\cdot13^2\cdot43\cdot79^3$\\[0.1cm]
137 &
$L_2({37}^2)$ & $2^3\cdot3^2\cdot 5 \cdot {19}\cdot {37}^2\cdot 137$  \\[0.1cm]&
$S_4({37})$ & $2^6\cdot3^4\cdot 5 \cdot {19}^2\cdot {37}^4\cdot 137$  \\[0.1cm]&
$L_4({37})$ & $2^7\cdot3^7\cdot 5\cdot 7 \cdot {19}^2\cdot {37}^6\cdot 67\cdot 137$  \\[0.1cm]&
$U_4({37})$ & $2^7\cdot3^4\cdot 5\cdot {19}^3\cdot 31 \cdot {37}^6\cdot 43\cdot 137$  \\[0.1cm]&
$L_3({137})$ & $2^7\cdot3\cdot 7\cdot {17}^2\cdot 23 \cdot {37}\cdot 73\cdot {137}^3$  \\[0.1cm]&
$L_3({37}^2)$ & $2^7\cdot3^4\cdot 5\cdot 7\cdot {19}^2\cdot 31 \cdot {37}^6\cdot 43\cdot 67\cdot {137}$  \\[0.1cm]&
$S_6({37})$ & $2^9\cdot3^7\cdot 5\cdot 7\cdot {19}^3\cdot 31 \cdot {37}^9\cdot 43\cdot 67\cdot {137}$  \\[0.1cm]&
$O_7({37})$ & $2^9\cdot3^7\cdot 5\cdot 7\cdot {19}^3\cdot 31 \cdot {37}^9\cdot 43\cdot 67\cdot {137}$  \\[0.1cm]&
$D_4({37})$ & $2^{12}\cdot3^9\cdot 5^2\cdot 7\cdot {19}^4\cdot 31 \cdot {37}^{12}\cdot 43\cdot 67\cdot {137}^2$  \\[0.1cm]
151 & $L_3(32)$ & $2^{15}\cdot3\cdot 7\cdot 11 \cdot {31}^2\cdot 151$  \\[0.1cm]
&$L_4(32)$ & $2^{30}\cdot3^2\cdot 5^2\cdot 7\cdot {11}^2\cdot {31}^3\cdot 41\cdot 151$  \\[0.1cm]
&$L_5(8)$ & $2^{30}\cdot3^4\cdot 5\cdot 7^3\cdot {13}\cdot {31}\cdot 73\cdot 151$  \\[0.1cm]
&$L_6(8)$ & $2^{45}\cdot3^7\cdot 5\cdot 7^4\cdot {13}\cdot 19\cdot
{31}\cdot {73}^2\cdot 151$
\\[0.1cm]
 181 & $L_2({19}^2)$ & $2^{3}\cdot3^2\cdot 5 \cdot {19}^2\cdot {181}$  \\[0.1cm]
&$S_4(19)$ & $2^{6}\cdot3^4\cdot5^2\cdot {19}^4 \cdot 181$  \\[0.1cm]
&$U_3(49)$ & $2^{6}\cdot3 \cdot5^4\cdot 7^6\cdot {13} \cdot 181$  \\[0.1cm]
&$U_4(19)$ & $2^{7}\cdot3^4\cdot5^3\cdot {7}^3\cdot {19}^6 \cdot 181$\\[0.1cm]
&$L_4(19)$ & $2^{7}\cdot3^7\cdot5^2\cdot {19}^6\cdot {127} \cdot 181$  \\[0.1cm]
&$L_3({19}^2)$ & $2^{7}\cdot3^4\cdot5^2\cdot 7^3\cdot {19}^6\cdot {127} \cdot 181$  \\[0.1cm]
&$S_6(19)$ & $2^{9}\cdot3^7\cdot5^3\cdot 7^3\cdot {19}^9\cdot {127} \cdot 181$  \\[0.1cm]
&$O_7(19)$ & $2^{9}\cdot3^7\cdot5^3\cdot 7^3\cdot {19}^9\cdot {127} \cdot 181$  \\[0.1cm]
&$D_8(19)$ & $2^{12}\cdot3^9\cdot5^4\cdot 7^3\cdot {19}^{12}\cdot {127} \cdot {181}^2$  \\[0.1cm]
&$^3D_4(7)$ & $2^{8}\cdot3^4\cdot7^{12}\cdot 13\cdot {19}^2\cdot {43}^2 \cdot 181$  \\[0.1cm]
&$L_2(7^6)$ & $2^{4}\cdot3^2\cdot5^2\cdot 7^6\cdot {13}\cdot 19\cdot {43} \cdot 181$  \\[0.1cm]
&$S_4(7^3)$ & $2^{8}\cdot3^4\cdot5^2\cdot 7^{12}\cdot 13\cdot {19}^2\cdot {43}^2 \cdot 181$  \\[0.1cm]
&$G_2(49)$ & $2^{10}\cdot3^3\cdot5^4 \cdot 7^{12}\cdot 13\cdot 19\cdot 43\cdot 181$  \\[0.1cm]
&$L_3(181)$ & $2^{5}\cdot3^4\cdot5^2\cdot 7\cdot {13}\cdot
79\cdot {139} \cdot {181}^3$ \\[0.1cm]
313 & $L_2(5^4)$ & $2^{4}\cdot 3\cdot 5^4\cdot 13\cdot 313$  \\
\hline
\end{tabular}
\end{center}
\begin{center}
{\bf Table 5}. {\em The non-abelian simple groups $S\in
{\cal{S}}_{p}$ except alternating ones (Continued).}

\vspace{.5cm}

\begin{tabular}{lll} \hline
$p$ & $S$ & $|S|$ \\[0.1cm]
\hline
313 &$S_4(25)$ & $2^{8}\cdot 3^2\cdot 5^8\cdot {13}^2\cdot 313$  \\[0.1cm]
&$L_3(313)$ & $2^{7}\cdot 3^2\cdot {13}^2\cdot 157\cdot {181}^2\cdot {313}^3$  \\[0.1cm]
&$^2D_4(5)$ & $2^{10}\cdot 3^4\cdot 5^{12}\cdot 7\cdot 13\cdot 31\cdot {313}$  \\[0.1cm]
&$S_8(5)$ & $2^{14}\cdot 3^5\cdot 5^{16}\cdot 7\cdot {13}^2\cdot 31\cdot {313}$  \\[0.1cm]
&$L_2({313}^2)$ & $2^{4}\cdot 3\cdot 5\cdot {13}\cdot 97\cdot {101}\cdot 157\cdot {313}^2$  \\[0.1cm]
&$S_4(313)$ & $2^{8}\cdot 3^2\cdot 5\cdot {13}^2\cdot 97\cdot 101\cdot {157}^2\cdot {313}^4$  \\[0.1cm]
&$D_5(5)$ & $2^{15}\cdot 3^5\cdot 5^{20}\cdot 7\cdot {11}\cdot {13}^2\cdot 31\cdot 71\cdot {313}$  \\[0.1cm]
&$L_4(25)$ & $2^{10}\cdot 3^4\cdot 5^{12}\cdot 7\cdot {13}^2\cdot 31\cdot 313$  \\[0.1cm]
&$L_4(313)$ & $2^{10}\cdot 3^4\cdot 5\cdot {13}^3\cdot 97\cdot 101\cdot {157}^2\cdot {181}^2\cdot {313}^6$  \\[0.1cm]
&$^3D_4(29)$ & $2^{6}\cdot 3^4\cdot {5}^2\cdot 7^2\cdot {13}^2\cdot 29^{12}\cdot 37\cdot 61\cdot {67}^2\cdot {271}^2\cdot {313}$  \\[0.1cm]
331 & $L_3(31)$ & $2^{7}\cdot 3^2\cdot 5^2\cdot {31}^3\cdot 331$  \\[0.1cm]
& $U_3(32)$ & $2^{15}\cdot 3^2\cdot {11}^2\cdot {31}\cdot 331$  \\[0.1cm]
& $L_2({31}^3)$ & $2^{5}\cdot 3^2\cdot 5\cdot 7^2\cdot 19\cdot {31}^3\cdot 331$  \\[0.1cm]
& $U_4(32)$ & $2^{30}\cdot 3^4\cdot 5^2\cdot {11}^3\cdot {31}^2 \cdot 41\cdot 331$  \\[0.1cm]
& $L_4(31)$ & $2^{13}\cdot 3^4\cdot 5^3\cdot {13}\cdot {31}^6 \cdot 37\cdot 331$  \\[0.1cm]
& $L_2(2^{15})$ & $2^{15}\cdot 3^2\cdot 7\cdot 11 \cdot {31}\cdot 151\cdot 331$  \\[0.1cm]
& $G_2(32)$ & $2^{30}\cdot 3^3\cdot 7\cdot {11}^2 \cdot {31}^2\cdot 151\cdot 331$  \\[0.1cm]
& $U_5(8)$ & $2^{30}\cdot 3^9\cdot 5\cdot 7^2\cdot 11\cdot 13\cdot 19\cdot 331$  \\[0.1cm]
& $U_6(8)$ & $2^{45}\cdot 3^{11}\cdot 5\cdot 7^3\cdot 11\cdot 13\cdot {19}^2\cdot 73\cdot 331$  \\[0.1cm]
& $L_3(2^{10})$ & $2^{30}\cdot 3^2\cdot 5^2\cdot 7\cdot {11}^2\cdot {31}^2\cdot 41\cdot 151\cdot 331$  \\[0.1cm]
& $S_6(32)$ & $2^{45}\cdot 3^4\cdot 5^2\cdot 7\cdot {11}^3\cdot {31}^3\cdot 41\cdot 151\cdot 331$  \\[0.1cm]
& $D_4(32)$ & $2^{60}\cdot 3^5\cdot 5^4\cdot 7\cdot {11}^4\cdot {31}^4 \cdot 41^2\cdot 151\cdot 331$  \\[0.1cm]
& $L_3({31}^2)$ & $2^{13}\cdot 3^2\cdot 5^2\cdot 7^2\cdot 19\cdot  {31}\cdot {31}^6\cdot 37\cdot 331$  \\[0.1cm]
& $O_7(31)$ & $2^{18}\cdot 3^4\cdot 5^3\cdot 7^2\cdot {13}\cdot 19\cdot {31}^9\cdot 37\cdot 331$  \\[0.1cm]
& $S_6(31)$ & $2^{18}\cdot 3^4\cdot 5^3\cdot 7^2\cdot {13}\cdot 19\cdot {31}^9\cdot 37\cdot 331$  \\[0.1cm]
& $D_4(31)$ & $2^{24}\cdot 3^5\cdot 5^4\cdot 7^2\cdot {13}^2\cdot 19\cdot {31}^{12} \cdot {37}^2\cdot 331$  \\[0.1cm]
& $^2D_5(8)$ & $2^{60}\cdot 3^{11}\cdot 5^2\cdot 7^4\cdot {11}\cdot {13}^2\cdot 17 \cdot 19\cdot 73\cdot 241\cdot 331$  \\[0.1cm]
& $L_5(64)$ & $2^{60}\cdot 3^9\cdot 5^2\cdot 7^4\cdot {11}\cdot {13}^2 \cdot 17\cdot 19\cdot 31\cdot 73\cdot 151\cdot 241\cdot 331$  \\[0.1cm]
& $S_{10}(8)$ & $2^{75}\cdot 3^{11}\cdot 5^2\cdot 7^5\cdot
{11}\cdot {13}^2 \cdot 17\cdot 19\cdot 31\cdot 73\cdot 151\cdot
241\cdot 331$
\\[0.1cm]
& $D_6(8)$ & $2^{90}\cdot 3^{14}\cdot 5^2\cdot 7^6\cdot {11}\cdot {13}^2 \cdot 17\cdot {19}^2\cdot 31\cdot {73}^2\cdot 151\cdot 241\cdot 331$  \\[0.1cm]
& $^2D_6(8)$ & $2^{90}\cdot 3^{11}\cdot 5^3\cdot 7^5\cdot {11}\cdot {13}^3 \cdot 17\cdot {19}\cdot 31\cdot 37\cdot {73}\cdot 151\cdot 241\cdot 331$  \\[0.1cm]
& $L_6(64)$ & $2^{90}\cdot 3^{11}\cdot 5^3\cdot 7^5\cdot {11}\cdot
{13}^3 \cdot 17\cdot {19}^2\cdot 31\cdot 37\cdot {73}^2
\cdot 109\cdot 151\cdot 241\cdot 331$  \\[0.1cm]
& $S_{12}(8)$ & $2^{108}\cdot 3^{14}\cdot 5^3\cdot 7^6\cdot
{11}\cdot {13}^3 \cdot 17\cdot {19}^2\cdot 31\cdot 37\cdot {73}^2
\cdot 109\cdot 151\cdot 241\cdot 331$ \\[0.1cm]
& $E_8(2)$ &   $2^{120}\cdot 3^{12}\cdot 5^5\cdot 7^4\cdot
{11}^2\cdot {13}^2 \cdot {17}^2\cdot {19}\cdot {31}^2\cdot 41\cdot
43\cdot 127\cdot 151\cdot 241\cdot 331$  \\[0.1cm]
421 & $L_2({29}^2)$ & $2^{3}\cdot 3\cdot 5\cdot 7\cdot {29}^2\cdot 421$  \\[0.1cm]
& $S_4(29)$ & $2^{6}\cdot 3^2\cdot 5^2\cdot 7^2\cdot {29}^4\cdot 421$  \\[0.1cm]
& $U_4(29)$ & $2^7\cdot 3^4\cdot 5^3\cdot 7^2\cdot {29}^6\cdot 271\cdot 421$  \\[0.1cm]
& $U_3(401)$ & $2^{6}\cdot 3^2\cdot 5^2\cdot {67}^2\cdot 127\cdot {401}^3\cdot 421$  \\[0.1cm]
\hline
\end{tabular}
\end{center}
\begin{center}

{\bf Table 5}. {\em The non-abelian simple groups $S\in
{\cal{S}}_{p}$ except alternating ones (Continued).}

\vspace{.5cm}

\begin{tabular}{lll} \hline
$p$ & $S$ & $|S|$ \\[0.1cm]
\hline
421& $L_4(29)$ & $2^{7}\cdot 3^2\cdot 5^2\cdot {7}^3\cdot 13\cdot {29}^6\cdot 67\cdot 421$  \\[0.1cm]
& $L_2({421}^2)$ & $2^3\cdot 3\cdot 5\cdot 7\cdot 13\cdot 17\cdot 211\cdot 401\cdot {421}^2$  \\[0.1cm]
& $S_4(421)$ & $2^6\cdot 3^2\cdot 5^2\cdot 7^2\cdot 13\cdot 17\cdot {211}^2\cdot 401\cdot {421}^4$  \\[0.1cm]
& $L_3({29}^2)$ & $2^7\cdot 3^2\cdot 5^2\cdot 7^2\cdot 13\cdot {29}^6\cdot 67\cdot 271\cdot 421$  \\[0.1cm]
& $S_6(29)$ & $2^9\cdot 3^4\cdot 5^3\cdot 7^3\cdot 13\cdot {29}^9\cdot 67\cdot 271\cdot 421$  \\[0.1cm]
& $O_7(29)$ & $2^9\cdot 3^4\cdot 5^3\cdot 7^3\cdot 13\cdot {29}^9\cdot 67\cdot 271\cdot 421$  \\[0.1cm]
& $D_4(29)$ & $2^{12}\cdot 3^5\cdot 5^4\cdot 7^4\cdot 13\cdot {29}^{12}\cdot 67\cdot 271\cdot 421^2$  \\[0.1cm]
& $U_3({29}^2$ & $2^5\cdot 3\cdot 5\cdot 7\cdot {29}^6\cdot 37\cdot 61\cdot 313\cdot {421}^2$  \\[0.1cm]
& $U_5(29)$ & $2^9\cdot 3^5\cdot 5^4\cdot 7^2\cdot 11\cdot {29}^{10}\cdot 31\cdot 271\cdot 401\cdot 421$  \\[0.1cm]
& $U_4(401)$ & $2^{11}\cdot 3^4\cdot 5^4\cdot 37\cdot 41\cdot 53\cdot {67}^3\cdot 127\cdot {401}^6\cdot 421$  \\[0.1cm]
& $U_6(29)$ &  $2^{11}\cdot 3^6\cdot 5^5\cdot 7^3\cdot 11\cdot 13\cdot {29}^{15}\cdot 31\cdot 67\cdot {271}^2\cdot 401\cdot 421$ \\[0.1cm]
& $L_2({29}^6)$ & $2^3\cdot 3^2\cdot 5\cdot 7\cdot 13\cdot {29}^6\cdot 37\cdot 61\cdot 67\cdot 271\cdot 313\cdot 421$  \\[0.1cm]
& $S_4({29}^3)$ &  $2^6\cdot 3^4\cdot 5^2\cdot 7^2\cdot {13}^2\cdot {29}^{12}\cdot 37\cdot 61\cdot {67}^2\cdot {271}^2\cdot 313\cdot 421$ \\[0.1cm]
& $G_2({29}^2)$ &  $2^8\cdot 3^3\cdot 5^2\cdot 7^2\cdot 13\cdot {29}^{12}\cdot 37\cdot 61\cdot 67\cdot 271\cdot 313\cdot {421}^2$ \\[0.1cm]
757 &  $L_3(27)$ & $2^{4}\cdot 3^9\cdot 7\cdot {13}^2\cdot 757$  \\[0.1cm]
&  $L_4(27)$ & $2^{7}\cdot 3^{18}\cdot 5\cdot 7^2\cdot {13}^3\cdot 73\cdot 757$  \\[0.1cm]
&  $L_2(3^9)$ & $2^2\cdot 3^9\cdot 7\cdot 13\cdot 19\cdot 37\cdot 757$  \\[0.1cm]
&  $G_2(27)$ & $2^6\cdot 3^{18}\cdot 7^2\cdot {13}^2\cdot 19\cdot 37\cdot 757$  \\[0.1cm]
&  $L_2({757}^2)$ & $2^3\cdot 3^3\cdot 5^2\cdot 7\cdot 73\cdot 157\cdot 379\cdot {757}^2$  \\[0.1cm]
&  $S_4(757)$ & $2^6\cdot 3^6\cdot 5^2\cdot 7^2\cdot 73\cdot 157\cdot {379}^2\cdot 757^4$  \\[0.1cm]
&  $E_6(3)$ & $2^{17}\cdot 3^{36}\cdot 5^2\cdot 7^2\cdot {11}^2\cdot {13}^3\cdot 41\cdot 73\cdot 757$  \\[0.1cm]
&  $L_3(3^6)$ & $2^{7}\cdot 3^{18}\cdot 5\cdot 7^2\cdot {13}^2\cdot 19\cdot 37\cdot 73\cdot 757$  \\[0.1cm]
&  $S_6(27)$ & $2^9\cdot 3^{27}\cdot 5\cdot 7^3\cdot {13}^3\cdot 19\cdot 37\cdot 73\cdot 757$  \\[0.1cm]
&  $O_7(27)$ & $2^9\cdot 3^{27}\cdot 5\cdot 7^3\cdot {13}^3\cdot 19\cdot 37\cdot 73\cdot 757$  \\[0.1cm]
&  $D_4(27)$ & $2^{12}\cdot 3^{36}\cdot 5^2\cdot 7^4\cdot {13}^4\cdot 19\cdot 37\cdot {73}^2\cdot 757$  \\[0.1cm]
&  $U_6(27)$ & $2^{13}\cdot 3^{45}\cdot 5\cdot 7^5\cdot {13}^3\cdot {19}^2\cdot 31\cdot {37}^2\cdot 61\cdot {73}\cdot 271\cdot 757$  \\[0.1cm]
\hline
\end{tabular}
\end{center}
\begin{center}
{\large \bf Acknowledgement}
\end{center}
The second author would like to thank the RIFS for the financial
support.

\end{document}